\documentclass[11pt,a4paper,twoside]{amsart}
\usepackage{amsmath}
\usepackage{amssymb}
\usepackage{amsthm}
\usepackage{geometry}
\usepackage{longtable}
\usepackage{array}
\usepackage{booktabs}
\usepackage{todonotes}
\usepackage{esint}
\usepackage{amssymb, url, color, pb-diagram, graphicx, amscd, pb-diagram,mathrsfs,amsfonts}
\usepackage[colorlinks=true, bookmarks=true, pdfstartview=FitH, pagebackref=true]{hyperref}
\usepackage{color}

\allowdisplaybreaks[3]

\newtheorem{theorem}{Theorem} [section]
\newtheorem{lemma}[theorem]{Lemma} 
\newtheorem{proposition}[theorem]{Proposition} 
\newtheorem{corollary}[theorem]{Corollary}

\newtheorem{remark}[theorem]{Remark}

\newtheorem{sketch proof}{Sketch of the proof}[section]

\let\ssection=\section\renewcommand{\section}{\setcounter{equation}{0}\ssection}

\newcommand{\dd}{\mathop{}\!\mathrm{d}}
\usepackage{cleveref}
\usepackage{etoolbox}
\crefname{theorem}{Theorem}{Theorems}

\crefname{lemma}{Lemma}{Lemmas}
\AtBeginEnvironment{lemma}{\crefalias{theorem}{lemma}}

\crefname{proposition}{Proposition}{Propositions}
\AtBeginEnvironment{proposition}{\crefalias{theorem}{proposition}}

\crefname{corollary}{Corollary}{Corollaries}
\AtBeginEnvironment{corollary}{\crefalias{theorem}{corollary}}  

\crefname{remark}{Remark}{Remarks}
\AtBeginEnvironment{remark}{\crefalias{theorem}{remark}}

\crefname{equation}{}{}
\crefname{section}{Section}{Section}

\begin{document}
	
	\title [Classification of solutions to Liouville equations and applications]{Classification of solutions to the Liouville equation with a nonlinear Neumann boundary condition and applications to sharp inequalities}


	\author{Xiaohan Cai and Abdolhakim SHOUMAN\\}
	\date{}
		\date{}

        \address{\noindent Xiaohan Cai\newline Shanghai Jiao Tong University, Shanghai,  200240, P. R. China. 
	}  
    \email{\href{mailto:xiaohancai@sjtu.edu.cn}{xiaohancai@sjtu.edu.cn} }
    
	\address{\noindent Abdolhakim Shouman\newline Université de Tours, Université d’Orléans, CNRS, IDP, UMR 7013, Tours, France
    }  
    \email{\href{mailto:hakim$\_$shouman@hotmail.com}{hakim$\_$shouman@hotmail.com} ,   \href{mailto:Shouman@univ-tours.fr}{shouman@univ-tours.fr}.}


\begin{abstract}
    In this paper, we study the Liouville equation with a nonlinear Neumann boundary condition
\begin{equation*}
    \begin{cases}
-\Delta u = Ke^{2u} & \text{in } \mathbb{B}^{2},\\[2mm]
\dfrac{\partial u}{\partial \nu}+\lambda = ke^{u} & \text{on } \partial\mathbb{B}^{2},
\end{cases}
\end{equation*}
where $K,k\in\mathbb{R}$, $\lambda\in(0,1]$ are constants,  and $\nu$ denotes the outward unit normal on $ \partial\mathbb{B}^{2}$.
We establish a classification of all smooth solutions to the equation.
Our approach is a boundary-adapted P-function method which treats all signs of $K$ and $k$ within a unified framework.
As applications of the classification result, we establish a family of sharp Sobolev-trace-type inequalities encompassing the classical Lebedev--Milin inequality, together with a corresponding deficit estimate.
\end{abstract}

\maketitle
\section{Introduction}

One culmination of classical surface theory is the uniformization theorem, at which topology, complex analysis, and geometric partial differential equations meet.
It asserts that every surface admits a conformal metric of constant Gauss curvature. The underlying principle also leads to the systematic study of prescribed Gauss curvature problem, and has had a lasting influence on geometric PDE and nonlinear analysis; see, for instance, \cite{KW71,KW74,Mos73,CY87, CGY93}.

The natural counterpart for compact surfaces with boundary has also garnered much study. Let $(M^2,g)$ be a compact surface with boundary, then the Gaussian curvature $K$ and the boundary geodesic curvature $k$ transform under a conformal transformation $\tilde{g}=e^{2u}g$ by
\begin{equation}
\label{eq. general conformal curvature equation}
    \begin{cases}
        -\Delta_g u+K_g=K_{\tilde{g}}e^{2u}, 
        &\text{in }M^2,\\[2mm]
        \displaystyle{\frac{\partial u}{\partial\nu}
        +k_g=k_{\tilde{g}}e^u}, 
        &\text{on }\partial M^2.
    \end{cases}
\end{equation}

The uniformization problem for compact surfaces with boundary was resolved in \cite{Osgood88,Bre02,Bre02Family}. In contrast, prescribing Gaussian and geodesic curvatures within a given conformal class is considerably more subtle.  The special cases in which either $K_{\tilde g}\equiv 0$ or $k_{\tilde g}\equiv 0$ have been extensively investigated; see, for instance, \cite{CY87,LL93,CSV25,CV19,CL96}.
When both curvatures are prescribed, the interaction between the interior and boundary nonlinearities gives rise to further analytic and variational difficulties. Consequently, their simultaneous prescription remains a challenging and active area of research; see \cite{CBR18,Rui24,LMR22,BLS25,LRR26} and the references therein.

Alongside the existence problem, a natural question is the uniqueness of the conformal metrics with constant Gauss curvature and geodesic curvature.
In higher dimensions, the analogous classification and uniqueness problem goes back to the foundational work of Obata \cite{Oba71} for closed manifolds and Escobar \cite{Esc88} for compact manifolds with convex boundary.
Motivated in part by these results, the classification  of solutions to equations of the form \cref{eq. general conformal curvature equation}, as well as their higher-dimensional counterparts, have been studied extensively. We refer the reader, for instance, to \cite{Osgood88,Esc90,HW06,Jim12,Wang2017,GHW21,GL25,Dou-Xu,Cai26} and the references therein.

When $(M^2,g)$ is the flat unit disk $\mathbb{B}^2$, the curvature transformation equation \cref{eq. general conformal curvature equation} simplifies to
\[
\begin{cases}
    -\Delta u=K_{\tilde{g}}e^{2u}
    &\text{in }\mathbb{B}^2,\\[2mm]
    \displaystyle{\frac{\partial u}{\partial\nu}
    +1=k_{\tilde{g}}e^u},
    &\text{on }\partial \mathbb{B}^2.
\end{cases}
\]
When $K_{\tilde g}$ and $k_{\tilde g}$ are constant, the classification of solutions is rather delicate, with difficulties arising from both the critical nature of the exponential nonlinearities and the non-compactness of the conformal automorphism group of $\mathbb{B}^2$.

Our first main result gives a complete classification of all smooth solutions to a slightly more general equation. 
\begin{theorem}
\label{ClassificationTh}
     Let $u$ be a smooth solution to the equation 
     \begin{equation} 
    \label{eq. conformal curvature equation}
    \begin{cases} -\Delta u = Ke^{2u} & \text{in } \mathbb{B}^2,\\[2mm] \dfrac{\partial u}{\partial \nu} + \lambda = ke^u & \text{on } \partial\mathbb{B}^2, \end{cases} 
    \end{equation}
     for three constants $K,k$ and $\lambda\in(0,1]$. Up to a scaling, we assume $K\in\{-1,0,+1\}$, then
     \[
     k\in\mathbb{R},\quad \text{if } K=1;\qquad
     k>0,\quad \text{if } K=0;\qquad
     \begin{cases}
         k\geq \sqrt{\lambda(2-\lambda)} &\text{if } 0<\lambda<1,\ K=-1,\\[2mm]
         k>1 &\text{if }
         \lambda=1,\qquad\, K=-1,
     \end{cases}
     \]
     and the solution is given by
     \[
     u(x)=
     \begin{cases}
         -\log
         \left(\frac{\epsilon^2+K |x|^2}{2\epsilon}\right)
         &\text{if }0<\lambda<1,\\[3mm]
         T_\phi\left(
         -\log
         \left(\frac{\epsilon^2+K |x|^2}{2\epsilon}\right)
         \right)
         &\text{if }\lambda=1,
     \end{cases}
     \]
     where $\epsilon$ is a constant depending explicitly on three constants $K,k$ and $\lambda$, $\phi\in \operatorname{Aut}(\mathbb{B}^2)$ is a M\"obius transformation, and 
     \begin{equation}
     \label{eq. Mobius action}
         T_\phi w:=w\circ\phi+\log |\phi'|
     \end{equation}
     is the two-dimensional M\"obius action  corresponding to the M\"obius transformation $\phi$.
\end{theorem}
\noindent Several special cases related to \cref{ClassificationTh} were previously known.
\begin{enumerate}
    \item In the conformally invariant case $\lambda=1$, conformal metrics with constant boundary geodesic curvature and positive, zero, or negative constant Gaussian curvature were classified in
\cite[Theorem 1.3]{LZ95}, 
\cite[Theorem 3.1]{Ou00}, 
and \cite[Theorem 4.3]{Zha03}, respectively, These works transform the equation into an equivalent problem on the upper half-plane and then apply the method of moving-planes. See also \cite{HW06} for an alternative proof based on complex-analytic methods.
\\

\item For general values of $\lambda$, several parameter regimes have also been treated.
In the positive-curvature regime, Dou and Xu \cite[Theorem~1.1]{Dou-Xu} established the classification for $K>0$ and $k=0$ through sophisticated integration by parts. Very recently, Dou-Hu-Peng \cite[Theorems 1.1 and 1.3]{DHP26} classified all solutions to \cref{eq. conformal curvature equation} for $K>0$, $k>0$, and $\lambda\in(0,2]$  by analyzing the holomorphic-lift representation of the solutions.
In the flat case $K=0$ and $k>0$, Wang \cite[Theorem~2]{Wang2017} obtained the classification via a pointwise maximum-principle argument; a closely related result had previously appeared in \cite{Osgood88}.
\end{enumerate}

We note that the higher-dimensional analogue of \cref{ClassificationTh} has been established by \cite{Esc90}. More precisely, Escobar proved the following classification result.
\begin{theorem}[\cite{Esc90}]
\label{thm. Escobar classification}
     Let $u \in C^{\infty}\left(\mathbb{B}^n\right)$ be a positive solution of the following equation
    \begin{equation}
    \label{eq. conformal scalar curvature equation}
        \begin{cases}\displaystyle{-\Delta u=\frac{n-2}{4(n-1)} R\, u^{\frac{n+2}{n-2}}}& \text { in } \mathbb{B}^n \\[4mm] 
    \displaystyle{\frac{\partial u}{\partial \nu}+\frac{n-2}{2} u=\frac{n-2}{2(n-1)} H u^{\frac{n}{n-2}}}& \text { on } \mathbb{S}^{n-1}
    \end{cases}
    \end{equation}
 where $R,H$ are two constants. Up to a scaling, we assume $R=\sigma n(n-1)$ for some $\sigma\in\{-1,0, +1\}$, then
 \[
 H\in\mathbb{R}\quad \text{if }\sigma=1;\qquad
 H>0\quad \text{if }\sigma=0;\qquad 
 H>n-1\quad \text{if }\sigma=-1,
 \]
 and there exists a Mobius transformation $\phi\in \operatorname{Aut}(\mathbb{B}^n)$, such that
 \[
 u(x)=T_\phi
 \left(\left(\frac{\epsilon^2+\sigma|x|^2}{2\epsilon}\right)^{-\frac{n-2}{2}}\right),
 \]
 where 
 $0<\epsilon:=\frac{H}{n-1}+\sqrt{\left(\frac{H}{n-1}\right)^2+\sigma}$, and
 \[
 T_\phi w:=(w\circ\phi) |\operatorname{det}D\phi|^{\frac{n-2}{2n}}.
 \]
\end{theorem}
Exploiting the conformal invariance of \cref{eq. conformal scalar curvature equation}, Escobar reduced the problem to the classification of solutions to the corresponding equation on the hemisphere $\mathbb{S}^n_+$, after which an Obata-type argument \cite{Oba71} completed the proof.

Inspired in part by \cite{Esc90}, we develop a boundary-adapted P-function method to prove \cref{ClassificationTh}.
Pioneered by L.~E. Payne, this method has proved particularly effective in rigidity problems; see, for example, \cite{Pay68,Wei71,CFP24,Cai26}.

A notable feature of our argument is that it treats all possible signs of $K$ and $k$ within a unified framework: their signs enter only in the final algebraic step. A comparison of our method with the pointwise maximum principle method used in \cite[Theorem 2]{Wang2017} is presented after the proof of \cref{ClassificationTh} (See \cref{RK comp with wang}). It is worth mentioning that our approach is also considerably simpler than those used to obtain the previously known classification results.
\vspace{1em}

For $\lambda<1$, the classification of solutions to \cref{eq. conformal curvature equation} goes beyond a formal extension of the uniqueness problem at the conformally invariant value $\lambda=1$:  it yields several further variational consequences.
This interplay between classification results and sharp functional inequalities echoes Escobar's work on the sharp Sobolev trace inequality in $\mathbb{R}_{+}^{n}$ \cite[Theorem 3.3]{Esc90}.

The starting observation is that, after a suitable normalization, \cref{eq. conformal curvature equation} arises as the Euler--Lagrange equation for one of the following functionals:
\begin{equation*}
        F_{\lambda,\alpha}(u)
        :=\frac{1}{4\pi}
        \int_{\mathbb{B}^2}|\nabla u|^2
        +\frac{\lambda}{2\pi}
        \int_{\partial \mathbb{B}^2}u
        -\frac{\alpha}{2}
        \log\left(
        \frac{1}{\pi}\int_{\mathbb{B}^2}e^{2u}
        \right)
        -(\lambda-\alpha)
        \log\left(
        \frac{1}{2\pi}
        \int_{\partial\mathbb{B}^2}e^{u}
        \right).
    \end{equation*}
    Our second main result is a family of sharp Sobolev-trace-type inequalities.
\begin{theorem}\label{thm. sharp Sobolev-trace inequality}
    For  $\alpha\in(-\infty,2)$, there holds
        \[
        F_{1,\alpha}(u)
        \geq -\frac{\alpha}{2}
        -\left(1-\frac{\alpha}{2}\right)
        \log\left(1-\frac{\alpha}{2}\right),\qquad 
        \forall u\in H^1(\mathbb{B}^2).
        \]
        Moreover, the equality holds if and only if
        \begin{equation*}
            u=T_{\phi}
            \left(
            -\log 
         \left(\frac{2-\alpha+\alpha|x|^2}{2}\right)
            \right)+c,\qquad \forall c\in\mathbb{R},
        \end{equation*}
        where  $T_\phi$ is the M\"obius action defined in \cref{eq. Mobius action}.
\end{theorem}
    At the endpoint $\alpha=2$, the infimum of $F_{1,2}$ equals $-1$ but is not attained. By contrast, if $\alpha>2$ or $\lambda>1$, the functional $F_{\lambda,\alpha}$ is unbounded below. This illustrates the sharpness of the stated ranges of $\alpha$ and $\lambda$; see \cref{prop. sharpness of alpha}.

    \cref{thm. sharp Sobolev-trace inequality} extends several known inequalities. 
\begin{enumerate}
    \item For $\alpha=0$, \cref{thm. sharp Sobolev-trace inequality} reduces to the classical Lebedev-Milin inequality \cite{LM51} (see also \cite[Chapter 5.1]{Dur01} and \cite[(4')]{Osgood88}) 
    \begin{equation}
    \label{eq. Lebedev-Milin ineq.}
        \frac{1}{4\pi}\int_{\mathbb{B}^2}|\nabla u|^{2}
    +\frac{1}{2\pi}\int_{\partial\mathbb{B}^2}u
    \geq \log\left(\frac{1}{2\pi}\int_{\partial\mathbb{B}^2}e^{u}\right),\quad 
    \forall u\in H^1(\mathbb{B}^2).
    \end{equation}
    
\item For $\alpha=1$, \cref{thm. sharp Sobolev-trace inequality} is an immediate consequence of the sharp Moser-Trudinger-Onofri inequality on $\mathbb{S}^2$ and has been presented explicitly in  \cite[Eq. (1.8)]{Dou-Xu}:
\begin{equation}
\label{eq. sharp Moser Trudinger Onofri on disk}
    \frac{1}{4\pi}\int_{\mathbb{B}^2}|\nabla u|^2
        +\frac{1}{2\pi}
        \int_{\partial \mathbb{B}^2}u
        \geq\frac{1}{2}
        \log\left(
        \frac{2}{e\pi}\int_{\mathbb{B}^2}e^{2u}
        \right),\quad 
         \forall u\in H^1(\mathbb{B}^2).
\end{equation}

\item For $\alpha\in(0,1)$, it is immediate to see that \cref{thm. sharp Sobolev-trace inequality} is nothing but a linear combination of the above two inequalities.
\end{enumerate}
    
    Compared with these known results, our \cref{thm. sharp Sobolev-trace inequality} is new for all $\alpha\in(-\infty,0)\cup(1,2]$. 
    Note that inequalities \cref{eq. Lebedev-Milin ineq.} and  \cref{eq. sharp Moser Trudinger Onofri on disk}, together with their extensions to general compact surfaces with boundary, are fundamental analytic tools in prescribing curvature problems; see, for instance, \cite{CBR18}. In light of this, we believe that the sharp inequalities established in \cref{thm. sharp Sobolev-trace inequality} hold potential for further applications in existence, compactness, and blow-up problems involving simultaneous interior and boundary curvature prescription.

    The proof of \cref{thm. sharp Sobolev-trace inequality} divides into two cases: $\alpha<1$ and $1\leq\alpha<2$. In the first case, the overall strategy is to implement the direct method to obtain a minimizer, then \cref{ClassificationTh} is used to determine the corresponding extremal value. When $1\leq\alpha<2$, however, the functional $F_{\lambda,\alpha}$ is no longer coercive on $H^1(\mathbb{B}^2)$, so the direct method cannot be applied directly. To overcome this obstacle, we establish an Aubin-type improvement of the Moser--Trudinger inequality \cite{Aub79}, which restores coercivity on  functions with vanishing barycenter; see \cref{lem. improved weak Moser on ball}.

    \vspace{1em}

    As a further consequence of the family of sharp inequalities, we derive a lower bound for the deficit of the classical Lebedev-Milin inequality.

    \begin{corollary}
    \label{cor. stability for LM inequality}
        For $u\in H^1(\mathbb{B}^2)$, define the isoperimetric ratio
        \begin{equation*}
            I(u):=\frac{\left(\int_{\partial\mathbb{B}^2}e^{u}\right)^2}{4\pi\int_{\mathbb{B}^2}e^{2u}}.
        \end{equation*}
       Then there holds
       \begin{equation*}
          F_{1,0}(u)=\frac{1}{4\pi}
        \int_{\mathbb{B}^2}|\nabla u|^2
        +\frac{1}{2\pi}
        \int_{\partial \mathbb{B}^2}u
        -\log\left(
        \frac{1}{2\pi}
        \int_{\partial\mathbb{B}^2}e^{u}
        \right)
        \geq \varphi(I(u)), \qquad \forall u\in H^1(\mathbb{B}^2),
       \end{equation*}
       where 
       \begin{equation*}
           \varphi(t):=t-1-\log t\geq0,\quad \forall t>0.
       \end{equation*}
       Moreover, equality holds if and only if 
       \begin{equation}
       \label{eq. minimizer}
           u(x)=T_\phi
           \left(
           -\log
        \left(\frac{1+q|x|^2}{1+q}\right)
           \right)+c,
       \end{equation}
       where $q>-1,\ c\in\mathbb{R}$ are constants, and $T_\phi$ is the M\"obius action defined in \cref{eq. Mobius action}.
    \end{corollary}
    \begin{remark}
        A key ingredient in establishing this deficit estimate is that \cref{thm. sharp Sobolev-trace inequality} holds for every $\alpha<2$.
        This full-range result, in turn, relies essentially on the classification of solutions to \cref{eq. conformal curvature equation} in all three curvature regimes: $K>0$, $K=0$, and $K<0$.
    \end{remark}
    When restricted to functions with vanishing  boundary trace, the deficit estimate reduces to:
    \begin{corollary}\label{cor. sharp Carleson-Chang ineq}
        For $u\in H_0^1(\mathbb{B}^2)$, there holds
        \begin{equation}
        \label{Carleson-Chang improv}
        \frac{1}{4\pi}
        \int_{\mathbb{B}^2}|\nabla u|^2
        +1
        \geq 
        \log\left(\frac{1}{\pi}\int_{\mathbb{B}^2}e^{2u}\right)
        +\frac{\pi}{\int_{\mathbb{B}^2}e^{2u}}.
        \end{equation}
        Moreover, the equality holds if and only if
        \[
        u(x)=-\log
        \left(\frac{1+q|x|^2}{1+q}\right),
        \] 
        for some constant $q>-1$. 
    \end{corollary}

    One intriguing point is that \cref{Carleson-Chang improv} sharpens the  strict inequality of Carleson and Chang \cite{CC86} 
        \begin{equation}
        \label{eq. Carleson Chang ineq}
            \frac{1}{4\pi}
        \int_{\mathbb{B}^2}|\nabla u|^2
        +1
        >
        \log\left(\frac{1}{\pi}\int_{\mathbb{B}^2}e^{2u}\right),\quad \forall u\in H_0^1(\mathbb{B}^2),
        \end{equation}
        Inequality \cref{eq. Carleson Chang ineq} is a key ingredient in Carleson and Chang's proof of the existence of extremals for the Dirichlet Moser--Trudinger inequality on $\mathbb{B}^2$ \cite{CC86}. More broadly, the concentration estimate derived from it has become an important tool in the compactness and blow-up analysis of critical exponential functionals; see, for example, \cite{Flu92,Yan16,IulaMancini2017}. Thus, our refinement \cref{Carleson-Chang improv} is of independent interest and holds potential  in related critical variational problems.
        
        \vspace{1em}

        {\bf AI disclosure statement:} Artificial intelligence tools were used during the preparation of this manuscript.
        Before any such tools were used, the authors had already proved \cref{ClassificationTh} and conceived the study of the functionals $F_{\lambda,\alpha}$ via the direct method. 
        ChatGPT-5.6 Sol was subsequently used to assist in elaborating the details of the proof of \cref{thm. sharp Sobolev-trace inequality} in the range $\alpha<1$. For $1\leq\alpha<2$, the authors used the model to interpret the relevant literature and guided the interaction toward the Aubin-type improvement stated in \cref{lem. improved weak Moser on ball}.                
         All the context and mathematical proof are entirely written by the authors, who take full responsibility for the manuscript.\\
         
        \noindent\textbf{Organization of the rest of the paper.} In \cref{Class section}, we prove  \cref{ClassificationTh}. In \cref{App results} and \cref{sec. functional ineq for large alpha}, we establish \cref{thm. sharp Sobolev-trace inequality} for $\alpha<1$ and $1\leq \alpha<2$, respectively.
        Finally, in \cref{sec. stability ineq for LM ineq},
        we prove \cref{cor. stability for LM inequality}.

\section{Classification results}\label{Class section}
In this section, we restate \cref{ClassificationTh} in a more precise form, treating separately the cases of positive, zero, and negative Gaussian curvature.
\begin{theorem}
    Let $u$ be a smooth solution to the equation
    \cref{eq. conformal curvature equation}
for some constants $K,k\in \mathbb{R}$ and  $0< \lambda\leq 1$. Then the following results hold.
\begin{enumerate}
     \item If $K>0$, 
        then, for any $k\in\mathbb{R}$, the solution is given by
        \begin{align*}
    u(x)=
    \begin{cases}
        -\log\left(\sqrt{K}
    \left(\frac{\epsilon_1}{2}+\frac{1}{2\epsilon_1}|x|^2\right)
    \right),&\text{if }0<\lambda<1,\\[3mm]
    \displaystyle{-\log\left(
    \sqrt{K}
    \frac{(1+\epsilon_2^2|a|^2)|x|^2
    +2(1+\epsilon_2^2)\langle x,a\rangle
    +\epsilon_2^2+|a|^2}{2\epsilon_2(1-|a|^2)}
    \right)}, &\text{if }\lambda=1,
    \end{cases}
        \end{align*}
where $\epsilon_1=\frac{k+\sqrt{k^2+\lambda(2-\lambda)K}}{\lambda\sqrt{K}}$, $\epsilon_2=\frac{k+\sqrt{k^2+K}}{\sqrt{K}}$ and $a\in\mathbb{B}^2$.
        \item If $K=0$, then $k>0$ and the solution is given by
        \begin{align*}
            u(x)=
            \begin{cases}
            -\log\frac{k}{\lambda} &\text{if }0<\lambda<1,\\[3mm]
                \displaystyle{-\log\left(
                k\frac{|a|^2|x|^2+2\langle x,a\rangle+1}{1-|a|^2}
                \right)} & \text{if }\lambda=1.
            \end{cases}
        \end{align*}
        for some $a\in \mathbb{B}^2$.
        \item If $K<0$, then 
        \begin{align*}
        \begin{cases}
            k\geq \sqrt{-\lambda(2-\lambda)K} \quad&\text{if }0<\lambda<1,\\[2mm]
            k>\sqrt{-K} &\text{if }\lambda=1,
        \end{cases}
        \end{align*}
        and the solution is given by
        \begin{align}\label{eq. solution for negative K small lambda}
            u(x)=
            \begin{cases}
               \displaystyle{ -\log\left(
                \sqrt{-K}
                \Big(\frac{\epsilon_3}{2}-\frac{1}{2\epsilon_3}|x|^2\Big)
                \right)} &\text{if } 0<\lambda<1,
                 k\in[\sqrt{-\lambda(2-\lambda)K},\sqrt{-K}),\\[3mm]
            \displaystyle{-\log\left(
                \sqrt{-K}
        \Big(\frac{\epsilon_4}{2}-\frac{1}{2\epsilon_4}|x|^2\Big)
                \right)} &\text{if } 0<\lambda<1,
                 k\in[\sqrt{-K},+\infty),\\[3mm]
                -\log\left(
                    \sqrt{-K}
    \frac{(\epsilon_5^2|a|^2-1)|x|^2
    +2(\epsilon_5^2-1)\langle x,a\rangle
    +\epsilon_5^2-|a|^2}{2\epsilon_5(1-|a|^2)}
                \right) &\text{if }\lambda=1.
            \end{cases}
        \end{align}
    where $\epsilon_3$ could be chosen from $\left\{
    \frac{k+\sqrt{k^2+\lambda(2-\lambda)K}}
    {\lambda\sqrt{-K}},
    \frac{k-\sqrt{k^2+\lambda(2-\lambda)K}}
    {\lambda\sqrt{-K}}   \right\}$, $\epsilon_4=\frac{k+\sqrt{k^2+\lambda(2-\lambda)K}}
    {\lambda\sqrt{-K}}$, $\epsilon_5=\frac{k+\sqrt{k^2+K}}{\sqrt{-K}}$ and $a\in\mathbb{B}^2$.
\end{enumerate}
\end{theorem}

Before presenting the details of the proof, we illustrate our overall strategy as follows. We choose an appropriate auxiliary P-function $P$ and derive an integral identity for $\Delta P$; see \cref{eq. integral of Delta P}. The resulting troublesome  boundary terms are then tackled by combining a weight function $w$ and $\Delta P$; see \cref{eq. first boundary term} and \cref{eq. second boundary term}. For the stated range of $\lambda$, this argument forces the P-function to be constant, from which the classification follows.

\begin{proof}
    {\bf Step 1:}
    Let $v:=e^{-u}$ and $\chi:=\frac{\partial v}{\partial\nu}|_{\partial\mathbb{B}^2}, f:=v|_{\partial\mathbb{B}^2}$. Then the equation \cref{eq. conformal curvature equation} is equivalent to 
\begin{align}\label{eq. equation of v}
\begin{cases}
    \Delta v=v^{-1}(K+|\nabla v|^2) &\text{in }\mathbb{B}^2,\\[2mm]
        \chi=\lambda f-k &\text{on }\partial\mathbb{B}^2.
\end{cases}
\end{align}
Define the P-function
\begin{align}\label{eq. P function}
    P:=v^{-1}(K+|\nabla v|^2)=\Delta v.
\end{align}
Then, by the Bochner formula, we have
    \begin{align*}
        v\Delta P+P\Delta v+2\langle\nabla P,\nabla v\rangle
        =&2\left(
        \left|\nabla^2v-\frac{\Delta v}{2}g_0\right|^2
        +\frac{1}{2}(\Delta v)^2
        +\langle\nabla\Delta v,\nabla v\rangle
        \right)\\
        =&2\left|\nabla^2v-\frac{\Delta v}{2}g_0\right|^2
        +P\Delta v+2\langle\nabla P,\nabla v\rangle.
    \end{align*}
    Rearranging it, we get
\begin{align}\label{eq. Delta P}
    \Delta P=2v^{-1}
    \left|\nabla^2 v-\frac{\Delta v}{2}g_0\right|^2\geq 0.
\end{align}
One key observation from \cref{eq. P function} is that the gradient of the P-function is closely related to the trace-less Hessian of $v$:
\begin{align}\label{eq. gradient P}
    \frac{1}{2}\nabla P
    =\frac{1}{2}\nabla\Delta v
    =v^{-1}
    \left(
    \nabla^2 v-\frac{\Delta v}{2}g_0
    \right)(\nabla v,\cdot).
\end{align}
Denote  $\nu$ to be the unit outer normal vector of $\partial \mathbb{B}^2$. Then by \cref{eq. equation of v} and \cref{eq. gradient P}, we have
\begin{align}
    \frac{1}{2}\int_{\mathbb{B}^2}\Delta P
    &= \frac{1}{2}\int_{\partial\mathbb{B}^2}
    \langle\nabla P,\nu\rangle
    =\int_{\partial\mathbb{B}^2}
    f^{-1}\left(
    \nabla^2 v-\frac{\Delta v}{2}g_0
    \right)(\nabla f+\chi\nu,\nu)\notag\\
    &= \int_{\partial\mathbb{B}^2}
    f^{-1}\langle\nabla_{\nabla f}\nabla v,\nu\rangle
    +f^{-1}\chi
    \left(
    \nabla^2 v-\frac{\Delta v}{2}g_0
    \right)(\nu,\nu)\notag
    \\
    &= \int_{\partial\mathbb{B}^2}
    f^{-1}\left(\langle\nabla f,\nabla \chi\rangle-|\nabla f|^2\right)
    +f^{-1}(\lambda f-k)
    \left(
    \nabla^2 v-\frac{\Delta v}{2}g_0
    \right)(\nu,\nu)\notag\\
    &= (\lambda-1)\int_{\partial\mathbb{B}^2}
    f^{-1}|\nabla f|^2
    +\lambda\int_{\partial\mathbb{B}^2}
    \left(
    \nabla^2 v-\frac{\Delta v}{2}g_0
    \right)(\nu,\nu)\notag\\
    &
    \hspace{4.2cm}-k\int_{\partial\mathbb{B}^2}
    f^{-1}\left(
    \nabla^2 v-\frac{\Delta v}{2}g_0
    \right)(\nu,\nu).\label{eq. integral of Delta P}
\end{align}

\noindent {\bf Step 2:} 
To tackle the last two terms in the above equation, we recall  \cite[Lemma 3.2]{Cai26}, which played a crucial role in classifying the solutions to the higher dimensional analogous equation of \cref{eq. conformal curvature equation}:
\begin{lemma}[\cite{Cai26}]\label{lem. Cai26}
    Let $(M^n,g)$ be a Riemannian manifold. Assume  that $(M^n,g)$ admits a smooth function $w$ such that $\nabla w$ is a closed conformal vector field on $M$, i.e. $\nabla ^2 w=\frac{\Delta w}{n}g$, then for any $v\in C^{\infty}(M)$ and constant $d>0$ there holds
    \begin{align*}
        \mathrm{div}\left(\nabla_{\nabla w}\nabla v-\frac{\Delta v}{d}\nabla w\right)
        =\left(\frac{1}{n}-\frac{1}{d}\right)\Delta v\Delta w
        -\left(1-\frac{1}{n}\right)\langle\nabla \Delta w,\nabla v\rangle
        +\left(1-\frac{1}{d}\right)\langle\nabla \Delta v,\nabla w\rangle.
    \end{align*}
\end{lemma}
\vspace{0.2 cm}
Consider the weight function
    \begin{align*}
        w(x):=\frac{1-|x|^2}{2},
    \end{align*}
    which satisfies
    \begin{align*}
        w|_{\partial \mathbb{B}^2}\equiv 0,\quad 
        \nabla w|_{\partial \mathbb{B}^2}
        =-\nu,\quad 
        \Delta w\equiv-2,\quad
        \nabla^2 w=\frac{\Delta w}{2}g_0.
    \end{align*}

By \cref{lem. Cai26} and \cref{eq. gradient P}, we have
\begin{align}\label{eq. auxiliary divergence}
    \mathrm{div}\left(\nabla_{\nabla w}\nabla v-\frac{\Delta v}{2}\nabla w\right)
    =\frac{1}{2}\langle\nabla \Delta v,\nabla w\rangle
    =v^{-1}\left(
    \nabla^2 v-\frac{\Delta v}{2}g_0
    \right)(\nabla v,\nabla w).
\end{align}
It follows from \cref{eq. auxiliary divergence} that
\begin{align*}
    &\mathrm{div}
    \left(v^{-1}
    \left(\nabla_{\nabla w}\nabla v-\frac{\Delta v}{2}\nabla w\right)\right)\\
    &= v^{-1}\mathrm{div}\left(\nabla_{\nabla w}\nabla v-\frac{\Delta v}{2}\nabla w\right)
    -v^{-2}
    \left\langle
    \nabla_{\nabla w}\nabla v-\frac{\Delta v}{2}\nabla w,\nabla v
    \right\rangle=0.
\end{align*}
Integrating it over $\mathbb{B}^2$, we get
\begin{align}\label{eq. first boundary term}
    \int_{\partial\mathbb{B}^2}f^{-1}
    \left(
    \nabla^2 v-\frac{\Delta v}{2}g_0
    \right)(\nu,\nu)
    =0.
\end{align}

Once again, \cref{lem. Cai26}, \cref{eq. gradient P} and \cref{eq. auxiliary divergence} imply that
\begin{align*}
    \frac{1}{2}w\Delta P
    &= w\,\mathrm{div}
    \left(
    v^{-1}\left(
    \nabla_{\nabla v}\nabla v-\frac{\Delta v}{2}\nabla v
    \right)
    \right)\\[2mm]
    & = \mathrm{div}\left(wv^{-1}\left(
    \nabla_{\nabla v}\nabla v-\frac{\Delta v}{2}\nabla v
    \right)\right)
    -v^{-1}
    \left\langle
    \nabla_{\nabla v}\nabla v-\frac{\Delta v}{2}\nabla v,\nabla w
    \right\rangle\\[2mm]
    &= \mathrm{div}\left(wv^{-1}\left(
    \nabla_{\nabla v}\nabla v-\frac{\Delta v}{2}\nabla v
    \right)\right)
    -\mathrm{div}\left(\nabla_{\nabla w}\nabla v-\frac{\Delta v}{2}\nabla w\right).
\end{align*}
Integrating it over $\mathbb{B}^2$, we get
\begin{align}\label{eq. second boundary term}
    \frac{1}{2}\int_{\mathbb{B}^2}w\Delta P
     =\int_{\partial \mathbb{B}^2}
    \left(
    \nabla^2 v-\frac{\Delta v}{2}g_0
    \right)(\nu,\nu).
\end{align}
Combining \cref{eq. Delta P}, \cref{eq. integral of Delta P}, \cref{eq. first boundary term} and \cref{eq. second boundary term}, we obtain our key integral identity:
\begin{align}\label{eq. key integral identity}
    0\leq
    \int_{\mathbb{B}^2}(1-\lambda w)
    v^{-1}\left|
    \nabla^2 v-\frac{\Delta v}{2}g_0
    \right|^2
    =(\lambda-1)\int_{\partial\mathbb{B}^2}f^{-1}|\nabla f|^2\leq 0.
\end{align}
 
\noindent{\bf Step 3:} 
If $0<\lambda<1$, then \cref{eq. key integral identity} implies that
\begin{align*}
    \nabla^2 v=\frac{\Delta v}{2}g_0\quad 
    \text{in }\mathbb{B}^2,\qquad\quad 
    v\equiv C\quad  \text{on }\partial\mathbb{B}^2.
\end{align*}
It follows from Ricci's identity that
\begin{align*}
    \nabla \Delta v
    =\mathrm{div}(\nabla^2v)
    =\mathrm{div}\left(\frac{\Delta v}{2}g_0\right)
    =\frac{\nabla \Delta v}{2}.
\end{align*}
Therefore, $\Delta v$ is a constant in $\mathbb{B}^2$. Combining with $v|_{\partial\mathbb{B}^2}\equiv C$, we could set 
\begin{align*}
    v(x)=r|x|^2+s
\end{align*}
for two constants $r,s$. Then, equation \cref{eq. equation of v} reduces to
\begin{equation}
\label{eq. quadratic equation for rs}
    \begin{cases}
        4rs=K, &\text{in } \mathbb{B}^2\\[2mm]
        (2-\lambda)r+k=\lambda s &\text{on } \partial \mathbb{B}^2.
    \end{cases}
\end{equation}
A formal calculation yields
\begin{equation}
    \begin{cases}
        r=\frac{-k+\sqrt{k^2+\lambda(2-\lambda)K}}{2(2-\lambda)}\\[2mm]
        s=\frac{k+\sqrt{k^2+\lambda(2-\lambda)K}}{2\lambda}
    \end{cases} \label{solution with plus sign}
\end{equation}
or
\begin{equation}
    \begin{cases}
        r=\frac{-k-\sqrt{k^2+\lambda(2-\lambda)K}}{2(2-\lambda)}\\[2mm]
        s=\frac{k-\sqrt{k^2+\lambda(2-\lambda)K}}{2\lambda}
    \end{cases}\label{solution with minus sign}.
\end{equation}
Now, we discuss what restriction on $K$ and $k$ the existence of a solution to \cref{eq. equation of v} would impose.
The positivity of $v$ and the nonnegativity of the discriminant could be  written as
\begin{align}\label{eq. conditions on positivity of v}
    s>0,\quad r+s>0,\quad 
    \text{and } \quad k^2\geq -\lambda(2-\lambda)K.
\end{align}
\begin{itemize}
    \item  If $K>0$, then \cref{eq. conditions on positivity of v} excludes  \cref{solution with minus sign}
for any $k\in\mathbb{R}$. Therefore, we achieve the classification of $v$:
\begin{align*}
    v(x)=\sqrt{K}\left(\frac{1}{2\epsilon}|x|^2+\frac{\epsilon}{2}\right),
\end{align*}
where $\epsilon=\frac{k+\sqrt{k^2+\lambda(2-\lambda)K}}{\lambda\sqrt{K}}$.

\item  If $K=0$, then the only meaningful solutions to \cref{eq. equation of v} exist only for $k>0$ and are given by
\begin{align*}
    r=0,\quad s=\frac{k}{\lambda}.
\end{align*}
Therefore, we achieve the classification of $v$:
\begin{align*}
    v(x)\equiv \frac{k}{\lambda}.
\end{align*}
\item 
If $K<0$, we note that $k$ must be positive. Otherwise, in any case of \cref{solution with plus sign} and \cref{solution with minus sign}, there always holds
\begin{align*}
    s<\frac{k+|k|}{2\lambda}\leq 0,
\end{align*}
which contradicts \cref{eq. conditions on positivity of v}. Then there are two sub-cases, depending on the value of $k$ due to \cref{eq. conditions on positivity of v}:
\begin{itemize}
    \item  If $k^2\in[-\lambda(2-\lambda)K,-K)$, then both \cref{solution with plus sign} and \cref{solution with minus sign} fulfill \cref{eq. conditions on positivity of v}. Therefore, the solution $v$ is given by
    \begin{align*}
        v(x)=\sqrt{-K}
        \left(\frac{\epsilon}{2}-\frac{1}{2\epsilon}|x|^2\right),
    \end{align*}
    where $\epsilon=
    \frac{k+\sqrt{k^2+\lambda(2-\lambda)K}}
    {\lambda\sqrt{-K}}$ or $\epsilon=
    \frac{k-\sqrt{k^2+\lambda(2-\lambda)K}}
    {\lambda\sqrt{-K}}$.
    \item If $k^2\in[-K,+\infty)$, then \cref{solution with minus sign} fails to fulfill $r+s>0$. Therefore, the solution $v$ is given by
    \begin{align*}
       v(x)=\sqrt{-K}
        \left(\frac{\epsilon}{2}-\frac{1}{2\epsilon}|x|^2\right),
    \end{align*}
    where $\epsilon=
    \frac{k+\sqrt{k^2+\lambda(2-\lambda)K}}
    {\lambda\sqrt{-K}}$.
\end{itemize}
\end{itemize}

\noindent{\bf Step 4:} Now we consider the case $\lambda=1$. \cref{eq. key integral identity} implies that
\begin{align*}
    \nabla^2 v=\frac{\Delta v}{2}g_0\quad 
    \text{in }\mathbb{B}^2,
\end{align*}
while $v|_{\partial\mathbb{B}^2}$ is not necessarily constant. As before, we could derive that $\Delta v$ is a constant in $\mathbb{B}^2$. Therefore, we set
\begin{align*}
    v(x)=r|x|^2+\langle x,\xi\rangle+s
\end{align*}
for two constants $r,s$ and $\xi\in\mathbb{R}^2$. Then, equation \cref{eq. equation of v} reduces to 
\begin{align*}
    \begin{cases}
        4rs=K+|\xi|^2 
        &\text{in }\mathbb{B}^2,\\[2mm]
        r+k=s&\text{on }\partial\mathbb{B}^2.
    \end{cases}
\end{align*}
Then the formal solutions are given by: 
\begin{align}
    &\begin{cases}
        r=\frac{-k+\sqrt{k^2+K+|\xi|^2}}{2}\\[2mm]
        s=\frac{k+\sqrt{k^2+K+|\xi|^2}}{2}
    \end{cases}\label{solution with plus sign lambda equal 1}\\
    \quad \text{and }\quad\nonumber\\
    &\begin{cases}
        r=\frac{-k-\sqrt{k^2+K+|\xi|^2}}{2}\\[2mm]
        s=\frac{k-\sqrt{k^2+K+|\xi|^2}}{2}
    \end{cases}.\label{solution with minus sign lambda equal 1}
\end{align}
We note that \cref{solution with minus sign lambda equal 1} could be excluded by the positivity of $v$ at the boundary point $-\frac{\xi}{|\xi|}$, i.e.
\begin{align*}
    r+s-|\xi|=v\left(-\frac{\xi}{|\xi|}\right)>0.
\end{align*}

\begin{itemize}

\item If $K>0$, then  \cref{solution with plus sign lambda equal 1} is well-defined for all $k\in\mathbb{R}$ and the corresponding solution $v$ is positive. We note that there exists a unique $a\in\mathbb{B}^2$ such that
\begin{align*}
    \xi=2\sqrt{k^2+K}\frac{a}{1-|a|^2}.
\end{align*}
Then we have
\begin{align*}
\begin{cases}
    r=\frac{-k+\sqrt{k^2+K}\frac{1+|a|^2}{1-|a|^2}}{2}
    =\sqrt{K}\frac{1+\epsilon^2 |a|^2}{2\epsilon(1-|a|^2)},\\[2mm]
    s=\frac{k+\sqrt{k^2+K}\frac{1+|a|^2}{1-|a|^2}}{2}
    =\sqrt{K}\frac{\epsilon^2+ |a|^2}{2\epsilon(1-|a|^2)},\\[2mm]
    \xi=\sqrt{K}\frac{(1+\epsilon^2)a}{\epsilon(1-|a|^2)},
\end{cases}
\end{align*}
where $\epsilon=\frac{k+\sqrt{k^2+K}}{\sqrt{K}}$.
Hence, the solution $v$ is given by
\begin{align*}
    v(x)=\sqrt{K}
    \frac{(1+\epsilon^2|a|^2)|x|^2
    +2(1+\epsilon^2)\langle x,a\rangle
    +\epsilon^2+|a|^2}{2\epsilon(1-|a|^2)}
\end{align*}
for some $a\in\mathbb{B}^2$.

    \item If $K=0$, applying the maximum principle to \cref{eq. conformal curvature equation}, we know that there exists no positive solution $v$ to \cref{eq. equation of v} for $k\leq 0$.  Therefore, by setting 
\begin{align*}
    \xi=2k\frac{a}{1-|a|^2}
\end{align*}
for some $a\in\mathbb{B}^2$, we derive the expression of $v$:
\begin{align*}
    v(x)=k\frac{|a|^2|x|^2+2\langle x,a\rangle+1}{1-|a|^2}.
\end{align*}
It is straightforward to see that $v$ is a positive solution to \cref{eq. equation of v} for any $k>0$.

\item If $K<0$, first note that
\begin{align*}
    0<v\left(-\frac{\xi}{|\xi|}\right)
    =r+s-|\xi|
    =\frac{k^2+K}{\sqrt{k^2+K+|\xi|^2}+|\xi|},
\end{align*}
so we have 
\begin{align}\label{restriction on k}
    k^2+K>0.
\end{align}
Next, we claim that $k>0$. Otherwise, it follows from \cref{restriction on k} that
\begin{align*}
    \left|-\frac{\xi}{2r}\right|
    =\frac{|\xi|}{\sqrt{k^2+K+|\xi|^2}-k}<1.
\end{align*}
Therefore,
\begin{align*}
    0<v\left(
    -\frac{\xi}{\sqrt{k^2+K+|\xi|^2}-k}
    \right)
    =\frac{K}{2(\sqrt{k^2+K+|\xi|^2}-k)}<0.
\end{align*}
That is a contradiction. As before, we set
\begin{align*}
    \xi=2\sqrt{k^2+K}\frac{a}{1-|a|^2}
\end{align*}
for some $a\in\mathbb{B}^2$.
Then we have
\begin{align*}
    \begin{cases}
        r=\frac{-k+\sqrt{k^2+K}\frac{1+|a|^2}{1-|a|^2}}{2}
    =\sqrt{-K}\frac{\epsilon^2 |a|^2-1}{2\epsilon(1-|a|^2)},\\[2mm]
    s=\frac{k+\sqrt{k^2+K}\frac{1+|a|^2}{1-|a|^2}}{2}
    =\sqrt{-K}\frac{\epsilon^2- |a|^2}{2\epsilon(1-|a|^2)},\\[2mm]
    \xi=\sqrt{-K}\frac{(\epsilon^2-1)a}{\epsilon(1-|a|^2)},
    \end{cases}
\end{align*}
where $\epsilon
=\frac{k+\sqrt{k^2+K}}{\sqrt{-K}}$.
Hence, the solution $v$ is given by
\begin{align*}
    v(x)=\sqrt{-K}
    \frac{(\epsilon^2|a|^2-1)|x|^2
    +2(\epsilon^2-1)\langle x,a\rangle
    +\epsilon^2-|a|^2}{2\epsilon(1-|a|^2)}
\end{align*}
for some $a\in\mathbb{B}^2$. 
\end{itemize}

This completes the classification of the solutions.
\end{proof}
\begin{remark}
    As we mentioned in the introduction, our treatment unifies all possible signs of $K$ and $k$.
    This uniformity is made possible by the key identity 
    \cref{eq. gradient P}.
\end{remark}
\begin{remark}\label{RK comp with wang}
    We note that, in the case $K=0$ and $k>0$, Wang \cite{Wang2017} classified all solutions to \cref{eq. conformal curvature equation} using a pointwise maximum-principle argument. Although Wang's method continues to work on more general background surface, it does not appear to extend to the classification of solutions to \cref{eq. conformal curvature equation} on $\mathbb{B}^2$ for $K\not=0$. Indeed, a direct adaptation of his argument yields only the inequality
\begin{equation}
\label{eq. wang argument}
S_K^2\leq -\frac{K}{2}S_K,
\end{equation}
 where 
    \[
    S_K:=ff''+
    \left(\lambda f-\frac{\chi}{2}\right)\chi
    -\frac{K}{2},
    \]
Here all quantities are evaluated at a boundary point where the P-function attains its maximum, and we follow the notation used in the proof of \cite[Theorem 6]{Wang2017}.

When $K=0$, \cref{eq. wang argument} forces $S_K=0$, which in turn implies that the outward normal derivative of the P-function at the maximum point is nonpositive. The Hopf lemma then forces the P-function to be constant.
When $K\not=0$, however, \cref{eq. wang argument} no longer forces $S_K$ to vanish and therefore does not provide the sign information needed to complete the argument.
\end{remark}
 
\section{Proof of \cref{thm. sharp Sobolev-trace inequality} for $\alpha<1$}\label{App results}

For $\lambda,\alpha\in\mathbb{R}$, define a family of functionals over $H^1(\mathbb{B}^2)$:
    \begin{equation}\label{eq. def of functional F}
        F_{\lambda,\alpha}(u)
        :=\frac{1}{4\pi}
        \int_{\mathbb{B}^2}|\nabla u|^2
        +\frac{\lambda}{2\pi}
        \int_{\partial \mathbb{B}^2}u
        -\frac{\alpha}{2}
        \log\left(
        \frac{1}{\pi}\int_{\mathbb{B}^2}e^{2u}
        \right)
        -(\lambda-\alpha)
        \log\left(
        \frac{1}{2\pi}
        \int_{\partial\mathbb{B}^2}e^{u}
        \right).
    \end{equation}
    
    First, we calculate the value of $F_{\lambda,\alpha}$ at some radial functions.
    \begin{lemma}\label{lem. F of wq}
        For $q>-1$, define
        \begin{equation}\label{eq. def of wq}
            w_q(x):=-\log
        \left(\frac{1+q|x|^2}{1+q}\right).
        \end{equation}
        Then there holds
        \[
         F_{\lambda,\alpha}(w_q)
        =\left(1-\frac{\alpha}{2}\right)
        \log(1+q)
        -\frac{q}{1+q}.
        \]
    \end{lemma}
    \begin{proof}
        A straightforward calculation yields
        \begin{equation}\label{eq. area of wq}
            \frac{1}{\pi}\int_{\mathbb{B}^2}e^{2w_q}
        =(1+q)^2\int_0^1\frac{2r\dd r}{(1+qr^2)^2}
        =(1+q)^2\int_0^1\frac{\dd t}{(1+qt)^2}
        =1+q,
        \end{equation}
        \begin{align}\label{eq. area of wq grad}
            \frac{1}{4\pi}
        \int_{\mathbb{B}^2}|\nabla w_{q}|^2
        =2q^2\int_0^1\frac{r^3\dd r}{(1+qr^2)^2}
        &=q\int_0^1
        \left(\frac{1}{1+qt}-\frac{1}{(1+qt)^2}\right)\dd t\\[3mm]
        &
        =\log(1+q)-\frac{q}{1+q}\notag.
        \end{align}
        Notice that $w_q=0$ on $\partial\mathbb{B}^2$, the proof finishes by inserting \cref{eq. area of wq} and \cref{eq. area of wq grad} into \cref{eq. def of functional F}.
    \end{proof}
\vspace{1em}

    By minimizing $F_{\lambda,\alpha}(w_q)$ among all $q>-1$, it is straightforward to see that
    \[
    F_{\lambda,\alpha}(w_q)
    \geq F_{\lambda,\alpha}(w_{\frac{\alpha}{2-\alpha}})
    =-\frac{\alpha}{2}
        -\left(1-\frac{\alpha}{2}\right)
        \log\left(1-\frac{\alpha}{2}\right),
    \]
    where 
    \[
    w_{\frac{\alpha}{2-\alpha}}(x)
    =-\log 
    \left(\frac{2-\alpha+\alpha|x|^2}{2}\right).
    \]
    In the following, we shall show that these are exactly the extremal functions of $F_{\lambda,\alpha}$.
\vspace{1em}

    To provide a transparent derivation of the family of functional inequalities, we deliberately avoid invoking the sharp inequalities such as \cref{eq. Lebedev-Milin ineq.,eq. sharp Moser Trudinger Onofri on disk}. This choice clarifies how our argument might be extended to general compact surfaces with boundary. The only standard analytic tool that we quote without proof is the classical Moser--Trudinger inequality; see \cite{Mos71}.
    
    \begin{theorem}[Moser-Trudinger inequality]\label{thm. Moser-Trudinger ineq}
\ 
    \begin{enumerate}
        \item Let $\mathbb{S}^2$ be the unit round sphere. There exists a universal constant $C_1$ such that, for any $u\in H^1(\mathbb{S}^2)$ with $\int_{\mathbb{S}^2}u=0$, there holds
        \[
        \int_{\mathbb{S}^2}\exp\left(
        \frac{4\pi u^2}{\int_{\mathbb{S}^2}
        |\nabla_{g_{\mathbb{S}^2}} u|_{g_{\mathbb{S}^2}}^2\dd \sigma_{g_{\mathbb{S}^2}}}
        \right)\dd \sigma_{g_{\mathbb{S}^2}}\leq C_1.
        \]
        \item Let $\Omega\subset \mathbb{R}^2$ be a smooth bounded domain. Then there exists a universal constant $C_2$ such that, for any $u\in H_0^1(\Omega)$, there holds
        \begin{equation}
        \label{eq. Moser-Trudinger on Omega}
            \int_{\Omega}\exp\left(
        \frac{4\pi u^2}{\int_{\Omega}|\nabla u|^2}
        \right)\leq C_2.
        \end{equation}
    \end{enumerate}
\end{theorem}
As a consequence, we readily obtain the following inequalities, which will be used repeatedly in the sequel. Although the arguments are standard, we include the details for the reader's convenience. We emphasize that the constants below are not optimal; nevertheless, these estimates suffice to establish the sharp inequalities later.
\begin{lemma}
    Let $\mathbb{B}^2$ be the unit disk in $\mathbb{R}^2$. Then there holds 
    \begin{align}
        \log\left(
        \frac{1}{\pi}\int_{\mathbb{B}^2}e^{2u}
        \right)
        \leq& \frac{1}{2\pi}\int_{\mathbb{B}^2}
        |\nabla u|^2
        +\frac{1}{\pi}
        \int_{\partial\mathbb{B}^2} u
        +\log\frac{eC_1}{8\pi},\quad 
        \forall u\in H^1(\mathbb{B}^2).
        \label{eq. weak Moser on ball}
        \\[3mm]
        \log\left(
        \frac{1}{\pi}\int_{\mathbb{B}^2}e^{2u}
        \right)
        \leq& \frac{1}{4\pi}\int_{\mathbb{B}^2}
        |\nabla u|^2+\log\frac{C_2}{\pi},\quad,\forall u\in H_0^1(\mathbb{B}^2).
        \label{eq. weak Moser for Dirichlet boundary value}
        \\[3mm]
        \log\left(
        \frac{1}{2\pi}
        \int_{\partial \mathbb{B}^2}e^{u}
        \right)
        \leq& \frac{1+\epsilon}{4\pi}\int_{\mathbb{B}^2}
        |\nabla u|^2
        +\frac{1}{2\pi}
        \int_{\partial\mathbb{B}^2} u
        +C_\epsilon,\quad \forall u\in H^1(\mathbb{B}^2),\qquad \forall \epsilon>0.\label{eq. weak LMO inequality}
    \end{align}
\end{lemma}
\begin{proof}
    We claim that
    \begin{equation}\label{eq. weak Moser on sphere}
        \log\left( \frac{1}{4\pi}\int_{\mathbb{S}^2}e^{2u}
        \dd \sigma_{g_{\mathbb{S}^2}}\right)
        \leq \frac{1}{4\pi}
        \int_{\mathbb{S}^2}
        |\nabla_{g_{\mathbb{S}^2}} u|_{g_{\mathbb{S}^2}}^2
        \dd \sigma_{g_{\mathbb{S}^2}}
        +\frac{1}{2\pi}\int_{\mathbb{S}^2}u
        \dd \sigma_{g_{\mathbb{S}^2}}
        +\log\frac{C_1}{4\pi},
        \quad \forall u\in H^1(\mathbb{S}^2).
    \end{equation}
    Indeed, fix $u\in H^1(\mathbb{S}^2)$ with $\int_{\mathbb{S}^2} u
    \dd \sigma_{g_{\mathbb{S}^2}}=0$ and $\int_{\mathbb{S}^2} 
    |\nabla_{g_{\mathbb{S}^2}} u|_{g_{\mathbb{S}^2}}^2
    \dd \sigma_{g_{\mathbb{S}^2}}>0$. By Young's inequality, we have
    \begin{align*}
        2u
        \leq \frac{4\pi u^2}{\int_{\mathbb{S}^2}
        |\nabla_{g_{\mathbb{S}^2}} u|_{g_{\mathbb{S}^2}}^2
        \dd \sigma_{g_{\mathbb{S}^2}}}
        +\frac{\int_{\mathbb{S}^2} 
        |\nabla_{g_{\mathbb{S}^2}} u|_{g_{\mathbb{S}^2}}^2
        \dd \sigma_{g_{\mathbb{S}^2}}}{4\pi}.
    \end{align*}
    Therefore, \cref{thm. Moser-Trudinger ineq} implies
    \begin{align*}
        \int_{\mathbb{S}^2} e^{2u}
        \dd \sigma_{g_{\mathbb{S}^2}}
        & \leq
        \exp\left(\frac{\int_{\mathbb{S}^2} 
        |\nabla_{g_{\mathbb{S}^2}} u|_{g_{\mathbb{S}^2}}^2
        \dd \sigma_{g_{\mathbb{S}^2}}}{4\pi}
        \right)
        \int_{\mathbb{S}^2}\exp\left(
        \frac{4\pi u^2}{\int_{\mathbb{S}^2}
        |\nabla_{g_{\mathbb{S}^2}} u|_{g_{\mathbb{S}^2}}^2\dd \sigma_{g_{\mathbb{S}^2}}}
        \right)\dd \sigma_{g_{\mathbb{S}^2}}\\[3mm]
        & \leq C_1
        \exp\left(\frac{\int_{\mathbb{S}^2} 
        |\nabla_{g_{\mathbb{S}^2}} u|_{g_{\mathbb{S}^2}}^2
        \dd \sigma_{g_{\mathbb{S}^2}}}{4\pi}
        \right)
    \end{align*}
    Taking logarithms gives
   \[
   \log\left( \frac{1}{4\pi}\int_{\mathbb{S}^2}e^{2u}
        \dd \sigma_{g_{\mathbb{S}^2}}\right)
        \leq \frac{1}{4\pi}
        \int_{\mathbb{S}^2}
        |\nabla_{g_{\mathbb{S}^2}} u|_{g_{\mathbb{S}^2}}^2
        \dd \sigma_{g_{\mathbb{S}^2}}
        +\log\frac{C_1}{4\pi}.
   \]
   Replacing $u$ by $u-\frac{1}{4\pi}\int_{\mathbb{S}^2} u
    \dd \sigma_{g_{\mathbb{S}^2}}$ proves our claim.
    
    Let $\phi(x):=\log\frac{2}{1+|x|^2}$.  It follows that $(\mathbb{S}^2_+,g_{\mathbb{S}^2_+})=(\mathbb{B}^2,e^{2\phi}g_0)$ and the upper hemisphere could be regarded as a conformal manifold over $\mathbb{B}^2$. Consider
    \[
    v:=u-\phi.
    \]
    Extending $v$ evenly over $\mathbb{S}^2$ gives a function $V\in H^1(\mathbb{S}^2)$. It follows from the conformal invariance that
    \begin{align*}
        &\int_{\mathbb{S}^2}e^{2V}
        \dd \sigma_{g_{\mathbb{S}^2}}
        =2\int_{\mathbb{S}^2}e^{2v}
        \dd \sigma_{g_{\mathbb{S}^2}}
        =2\int_{\mathbb{B}^2}e^{2u}
        ,\\[2mm]
        &\int_{\mathbb{S}^2}
        |\nabla_{g_{\mathbb{S}^2}} V|_{g_{\mathbb{S}^2}}^2
        \dd \sigma_{g_{\mathbb{S}^2}}
        =2\int_{\mathbb{S}^2_+}
        |\nabla_{g_{\mathbb{S}^2}} v|_{g_{\mathbb{S}^2}}^2
        \dd \sigma_{g_{\mathbb{S}^2}}
        =2\int_{\mathbb{B}^2}|\nabla v|^2
        ,\\[2mm]
        &\int_{\mathbb{S}^2} V
        \dd \sigma_{g_{\mathbb{S}^2}}
        =2\int_{\mathbb{S}_+^2}v
        \dd \sigma_{g_{\mathbb{S}^2}}
        =2\int_{\mathbb{B}^2}ve^{2\phi}.
    \end{align*}
    Combining these with \cref{eq. weak Moser on sphere}, we get
    \[
    \log\left( \frac{1}{2\pi}\int_{\mathbb{B}^2}e^{2u}
        \right)
        \leq \frac{1}{2\pi}
        \int_{\mathbb{B}^2}
        |\nabla v|^2
        +\frac{1}{\pi}\int_{\mathbb{B}^2}
        ve^{2\phi}
        +\log\frac{C_1}{4\pi}.
    \]
    Noting that $-\Delta\phi=e^{2\phi}$ in $\mathbb{B}^2$ and $\phi=0,\ \frac{\partial\phi}{\partial\nu}=-1$ on $\partial\mathbb{B}^2$, the right hand side of this inequality could be written as
    \begin{align*}
        \frac{1}{2\pi}
        \int_{\mathbb{B}^2}
        |\nabla v|^2
        +\frac{1}{\pi}\int_{\mathbb{B}^2}
        ve^{2\phi}
        &= \frac{1}{2\pi}\int_{\mathbb{B}^2}
        |\nabla v|^2-2v\Delta\phi\\[2mm]
        &= \frac{1}{2\pi}\int_{\mathbb{B}^2}
        |\nabla v|^2+2\langle\nabla v,\nabla\phi\rangle
        +\frac{1}{\pi}
        \int_{\partial\mathbb{B}^2}v\\[2mm]
        &= \frac{1}{2\pi}\int_{\mathbb{B}^2}
        |\nabla u|^2-|\nabla\phi|^2
        +\frac{1}{\pi}
        \int_{\partial\mathbb{B}^2}u\\[2mm]
        &= \frac{1}{2\pi}\int_{\mathbb{B}^2}
        |\nabla u|^2
        +\frac{1}{\pi}
        \int_{\partial\mathbb{B}^2}u
        -2\log 2+1.
    \end{align*}
    This proves \cref{eq. weak Moser on ball}.

    Using the second assertion in \cref{thm. Moser-Trudinger ineq}, the same proof as that of \cref{eq. weak Moser on sphere}  yields \cref{eq. weak Moser for Dirichlet boundary value}.

    By divergence theorem, for a smooth function $u$,
    \begin{align*}
        \int_{\partial\mathbb{B}^2}e^u
    =\int_{\mathbb{B}^2}\operatorname{div}(e^ux)
    =&2\int_{\mathbb{B}^2}e^u
    +\int_{\mathbb{B}^2}e^u\langle x,\nabla u\rangle\notag\\
    \leq& 2\sqrt{\pi}
    \left(
    \int_{\mathbb{B}^2} e^{2u}
    \right)^{\frac{1}{2}}
    +\left(
    \int_{\mathbb{B}^2} e^{2u}
    \right)^{\frac{1}{2}}
    \left(
    \int_{\mathbb{B}^2} |\nabla u|^2
    \right)^{\frac{1}{2}}.
    \end{align*}
    Equivalently,
    \begin{equation}\label{eq. isoperimetric upper bound}
        \left(
        \frac{1}{2\pi}
        \int_{\partial \mathbb{B}^2}e^{u}
        \right)^2
        \leq \left(
        \frac{1}{\pi}\int_{\mathbb{B}^2}e^{2u}
        \right)
        \left(1+
        \left(\frac{1}{4\pi}
        \int_{\mathbb{B}^2}|\nabla u|^2\right)^{\frac{1}{2}}\right)^2.
    \end{equation}
    It follows from \cref{eq. weak Moser on ball} and \cref{eq. isoperimetric upper bound} that
    \begin{align*}
        2\log\left(
        \frac{1}{2\pi}
        \int_{\partial \mathbb{B}^2}e^{u}
        \right)
        & \leq 
        \frac{1}{2\pi}\int_{\mathbb{B}^2}
        |\nabla u|^2
        +\frac{1}{\pi}
        \int_{\partial\mathbb{B}^2} u
        +\log\frac{eC_1}{8\pi}
        +2\log
        \left(1+
        \left(\frac{1}{4\pi}
        \int_{\mathbb{B}^2}|\nabla u|^2\right)^{\frac{1}{2}}\right)\\[2mm]
        & \leq 
        \frac{1+\epsilon}{2\pi}\int_{\mathbb{B}^2}
        |\nabla u|^2
        +\frac{1}{\pi}
        \int_{\partial\mathbb{B}^2} u
        +C_\epsilon.
    \end{align*}
    where in the last line we used, for any $\epsilon>0$, there exists a constant $C_\epsilon>0$, such that
    \[
    \log(1+t^{\frac{1}{2}})
    \leq \epsilon t+C_\epsilon,\quad \forall t\geq 0.
    \]
    This proves \cref{eq. weak LMO inequality}.
\end{proof}

    Next, we prove a compactness lemma, which plays a fundamental role in later minimizing arguments.
    \begin{lemma}\label{lem. compactness in H1}
        Assume $\{u_j\}$ is a bounded sequence in $H^1(\mathbb{B}^2)$. Then, up to a subsequence, there exists $u_\infty\in H^1(\mathbb{B}^2)$, such that
        \[
        u_j\rightharpoonup u_\infty \quad \text{in } H^1(\mathbb{B}^2).
        \]
        Moreover, for any $q\in\mathbb{R}$, there holds
        \begin{align*}
            &e^{qu_j}\to e^{qu_\infty}
            \quad \text{in }L^1(\mathbb{B}^2),\\
            &e^{qu_j}\to e^{qu_\infty}
            \quad 
            \text{in }L^1(\partial\mathbb{B}^2).
        \end{align*}
    \end{lemma}
    \begin{proof}
        Notice that $||u_j||_{H^1(\mathbb{B}^2)}$ is uniformly bounded, we know that, up to a subsequence, they weakly converge to a function 
        $u_\infty\in H^1(\mathbb{B}^2)$.

        By Cauchy-Schwarz inequality and Sobolev trace inequality, we get
        \begin{equation}\label{eq. uniform upper bound for boundary integral}
            \left|\int_{\partial \mathbb{B}^2}u_j\right|
        \leq \sqrt{2\pi}
        ||u_j||_{L^2(\partial \mathbb{B}^2)}
        \leq C||u_j||_{H^1(\mathbb{B}^2)}
        \leq C.
        \end{equation}
        Combining with \cref{eq. weak Moser on ball}, it follows that
        \begin{equation}\label{eq. uniform upper bound for e2u}
            \log\left(
        \frac{1}{\pi}\int_{\mathbb{B}^2}e^{2u_j}
        \right)\leq C,\quad \forall j\geq 1.
        \end{equation}
        Likewise, the same upper bound also holds for $u_\infty$.

        Similarly, by \cref{eq. weak LMO inequality} and \cref{eq. uniform upper bound for boundary integral}, we get
        \begin{equation}\label{eq. uniform upper bound for eu}
            \log\left(
        \frac{1}{2\pi}
        \int_{\partial\mathbb{B}^2}e^{u_j}
        \right)\leq C,\quad \forall j\geq 1.
        \end{equation}

        In dimension two, there are two compact embeddings
        \[
        H^1(\mathbb{B}^2)\Subset L^p(\mathbb{B}^2),\quad 
        H^1(\mathbb{B}^2)\hookrightarrow
        H^{\frac{1}{2}}(\mathbb{B}^2)
        \Subset L^p(\partial \mathbb{B}^2),\qquad 
        \forall 1\leq p<+\infty.
        \]
        Hence 
        \begin{equation}\label{eq. strong convergence in Lp}
            u_j\to u_\infty \quad \text{in }L^p(\mathbb{B}^2),\qquad 
        u_j\to u_\infty \quad \text{in }L^p(\partial \mathbb{B}^2),
        \qquad 
        \forall 1\leq p<+\infty.
        \end{equation}
        Notice that, by mean-value theorem,
        \[
        |e^{qs}-e^{qt}|
        \leq |q|(e^{qs}+e^{qt})|s-t|,\quad 
        \forall s,t\in\mathbb{R}.
        \]
        It follows from  Cauchy-Schwarz inequality, \cref{eq. uniform upper bound for e2u} and \cref{eq. strong convergence in Lp} that
        \begin{align*}
            ||e^{qu_j}-e^{qu_\infty}||
            _{L^1(\mathbb{B}^2)}
            &\leq |q|\,
            \big|\big||u_j-u_\infty|(e^{qu_j}+e^{qu_\infty})\big|\big|
            _{L^1(\mathbb{B}^2)}\\[2mm]
            &\leq |q|\, 
            ||u_j-u_\infty||_{L^2(\mathbb{B}^2)}
            \left(
            ||e^{qu_j}||_{L^2(\mathbb{B}^2)}
            +||e^{qu_\infty}||_{L^2(\mathbb{B}^2)}
            \right)\\[2mm]
            &\leq C|q|\, 
            ||u_j-u_\infty||
            _{L^2(\mathbb{B}^2)}  \to 0,
            \quad \text{as }j\to+\infty.
        \end{align*}
        Similarly, by \cref{eq. uniform upper bound for eu} and \cref{eq. strong convergence in Lp}, one obtains
        $||e^{qu_j}- e^{qu_\infty}||_{L^1(\partial \mathbb{B}^2)}\to 0.$
    \end{proof}
We are now in a position to prove  \cref{thm. sharp Sobolev-trace inequality} when $\alpha<1$. 
    \begin{proof}[Proof of \cref{thm. sharp Sobolev-trace inequality} when $\alpha<1$]
        Since $\alpha<1$, we can choose $\lambda\in(\max\{\alpha,0\},1)$.
        
        {\bf Step 1:} We claim that $F_{\lambda,\alpha}$ is coercive and is bounded from below.     
        Indeed, if $0\leq \alpha<1$, then we rewrite $F_{\lambda,\alpha}$ as
        \begin{align*}
            F_{\lambda,\alpha}(u)
        &= \frac{1-\lambda-\epsilon(\lambda-\alpha)}{4\pi}
        \int_{\mathbb{B}^2}|\nabla u|^2
        +\frac{\alpha}{2}\left(
        \frac{1}{2\pi}
        \int_{\mathbb{B}^2}|\nabla u|^2
        +\frac{1}{\pi}
        \int_{\partial \mathbb{B}^2}u
        -\log\left(\frac{1}{\pi}\int_{\mathbb{B}^2}e^{2u}\right)
        \right)\\
        & +(\lambda-\alpha)
        \left(
        \frac{1+\epsilon}{4\pi}
        \int_{\mathbb{B}^2}|\nabla u|^2
        +\frac{1}{2\pi}
        \int_{\partial \mathbb{B}^2}u
        -\log\left(\frac{1}{2\pi}
        \int_{\partial\mathbb{B}^2}e^{u}\right)
        \right).
        \end{align*}
        By \cref{eq. weak Moser on ball} and \cref{eq. weak LMO inequality}, we derive
        \[
        F_{\lambda,\alpha}(u)
        \geq \frac{1-\lambda-\epsilon(\lambda-\alpha)}{4\pi}
        \int_{\mathbb{B}^2}|\nabla u|^2
        -C_\epsilon.
        \]
        Fix $0<\epsilon\ll 1$, such that $1-\lambda-\epsilon(\lambda-\alpha)>0$. Then we get that $F_{\lambda,\alpha}$ is coercive and is bounded from below.

        If $\alpha<0$, then we rewrite $F_{\lambda,\alpha}$ as
        \begin{align*}
            F_{\lambda,\alpha}(u)
        =\frac{1-\lambda-\epsilon\lambda}{4\pi}
        \int_{\mathbb{B}^2}|\nabla u|^2
        +&\lambda
        \left(
        \frac{1+\epsilon}{4\pi}
        \int_{\mathbb{B}^2}|\nabla u|^2
        +\frac{1}{2\pi}
        \int_{\partial \mathbb{B}^2}u
        -\log\left(\frac{1}{2\pi}
        \int_{\partial\mathbb{B}^2}e^{u}\right)
        \right)\\
        &+\frac{\alpha}{2}\log
        \left(
        \frac{\left(\frac{1}{2\pi}
        \int_{\partial\mathbb{B}^2}
        e^{u}\right)^2}{\frac{1}{\pi}\int_{\mathbb{B}^2}e^{2u}}
        \right).
        \end{align*}
        It follows from \cref{eq. weak LMO inequality} and \cref{eq. isoperimetric upper bound} that
        \[
        F_{\lambda,\alpha}
        \geq \frac{1-\lambda-\epsilon\lambda}{4\pi}
        \int_{\mathbb{B}^2}|\nabla u|^2
        +\alpha\log\left(
        1+\left(\frac{1}{4\pi}\int_{\mathbb{B}^2}|\nabla u|^2\right)^{\frac{1}{2}}
        \right)
        -C_\epsilon.
        \]
        After fixing $\epsilon$ such that $1-\lambda-\epsilon\lambda>0$, we also get that $F_{\lambda,\alpha}$ is coercive and is bounded from below. This completes the proof of our claim.

        {\bf Step 2:} Let $\{u_j\}$ be a minimizing sequence of $F_{\lambda,\alpha}$, then the coercivity of $F_{\lambda,\alpha}$ implies a uniform upper bound of the Dirichlet energy
        \[
        \int_{\mathbb{B}^2}|\nabla u_j|^2\leq C.
        \]
        Notice that $F_{\lambda,\alpha}$ is invariant under shifting $u\mapsto u+c$, after subtracting a constant, we could assume
        \[
        \int_{\mathbb{B}^2}u_j=0.
        \]
        Then it follows from the Poincar\'e's inequality that
        \[
        ||u_j||_{H^1(\mathbb{B}^2)}\leq C.
        \]
        Then by \cref{lem. compactness in H1}, there exists $u_\infty\in H^1(\mathbb{B}^2)$ with $e^{2u_j}\to e^{2u_\infty}$ in $L^1(\mathbb{B}^2)$, $e^{u_j}\to e^{u_\infty}$ in $L^1(\partial\mathbb{B}^2)$. Combining these with the lower semi-continuity of the semi-norm of $H^1(\mathbb{B}^2)$, we deduce that $u_\infty$ is a minimizer of $F_{\lambda,\alpha}$. 
        
        A direct calculation of the Euler-Lagrange equation of the functional $F_{\lambda,\alpha}$ \cref{eq. def of functional F} yields
        \begin{equation}\label{eq. equation of minimizer}
            \begin{cases}
               \displaystyle{ -\Delta u_\infty
                =\frac{2\pi\alpha}{\int_{\mathbb{B}^2}e^{2u_\infty}}e^{2u_\infty}} &\text{in }\mathbb{B}^2,\\[5mm]
                \displaystyle{\frac{\partial u_\infty}{\partial\nu}
                +\lambda
                =\frac{2\pi(\lambda-\alpha)}{\int_{\partial \mathbb{B}^2}e^{u_\infty}}e^{u_\infty}}
                &\text{on }\partial \mathbb{B}^2.
            \end{cases}
        \end{equation}
        This is exactly the equation \cref{eq. conformal curvature equation} of \cref{ClassificationTh} with
        \begin{equation*}\label{eq. determine K and k}
            K=\frac{2\pi\alpha}{\int_{\mathbb{B}^2}e^{2u_\infty}},\quad k=\frac{2\pi(\lambda-\alpha)}{\int_{\partial \mathbb{B}^2}e^{u_\infty}}.    
        \end{equation*}
        
        {\bf Step 3:}
        By \cref{ClassificationTh}, noting that $0<\lambda<1$, the solution $u_\infty$ must be radially symmetric and is of the form
        \[
        u_\infty=w_{q_\infty}+c_\infty,
        \]
        where $w_{q_\infty}$ is defined in \cref{eq. def of wq} and $q_\infty>-1$, $c_\infty\in\mathbb{R}$ are two constants depending on $\alpha$ and $\lambda$.
        Note that a direct calculation gives
        \[
        -\Delta w_{q_\infty}
        =\frac{4q_\infty}{(1+q_\infty)^2}
        e^{2w_{q_\infty}},
        \]
        we get
        \[
        -\Delta u_\infty
        =\frac{4q_\infty}{(1+q_\infty)^2}e^{-2c_\infty}
        e^{2u_\infty}.
        \]
        Combining this with \cref{eq. equation of minimizer}, we could determine $q_\infty$ as follows:
        \[
        \alpha=\frac{1}{2\pi}
        \frac{4q_\infty}{(1+q_\infty)^2}
        \int_{\mathbb{B}^2}
        e^{2u_\infty}e^{-2c_\infty}
        =\frac{2q_\infty}{1+q_\infty}.
        \]
        where we used \cref{eq. area of wq} in the last equality. Equivalently,
        \begin{equation}\label{eq. determine q infty}
            q_\infty=\frac{\alpha}{2-\alpha}.
        \end{equation}
        By shifting-invariance of $F_{\lambda,\alpha}$, \cref{eq. determine q infty} and \cref{lem. F of wq}, we get
        \[
        \inf_{u\in H^1(\mathbb{B}^2)}
        F_{\lambda,\alpha}(u)
        =F_{\lambda,\alpha}(u_\infty)
        =F_{\lambda,\alpha}
        (w_{\frac{\alpha}{2-\alpha}})
        =-\frac{\alpha}{2}
        -\left(1-\frac{\alpha}{2}\right)
        \log \left(1-\frac{\alpha}{2}\right).
        \]
        Letting $\lambda$ go to $1$ proves
        \[
        F_{1,\alpha}(u)\geq 
         -\frac{\alpha}{2}
        -\left(1-\frac{\alpha}{2}\right)
        \log \left(1-\frac{\alpha}{2}\right)
        =F_{1,\alpha}(w_{\frac{\alpha}{2-\alpha}}).
        \]
        Finally, by \cref{ClassificationTh}, any critical points of $F_{1,\alpha}$ is of the form \cref{eq. minimizer} and 
        the proof finishes.
    \end{proof}

    \section{Proof of \cref{thm. sharp Sobolev-trace inequality} for $1\leq \alpha<2$}
    \label{sec. functional ineq for large alpha}
    In this section, we prove \cref{thm. sharp Sobolev-trace inequality} for $\alpha\in[1,2)$.
    \begin{theorem}\label{thm. functional ineq for alpha larger than 1}
        For  $\alpha\in[1,2)$, there holds
        \[
        F_{1,\alpha}(u)
        \geq -\frac{\alpha}{2}
        -\left(1-\frac{\alpha}{2}\right)
        \log\left(1-\frac{\alpha}{2}\right),\qquad 
        \forall u\in H^1(\mathbb{D}^2).
        \]
        Moreover, the equality holds if and only if
        \begin{equation*}
            u=T_{\phi_a}w_{\frac{\alpha}{2-\alpha}}+c,\qquad \forall c\in\mathbb{R},
        \end{equation*}
        where  
        \[
        \phi_a(z):=\frac{z+a}{1+\bar{a}z},\quad a\in\mathbb{B}^2,
        \]
        is a M\"obius transformation, and 
        \[
        T_\phi w:=w\circ\phi+\log |\phi'|.
        \]
    \end{theorem}
    When $\alpha\in[1,2)$, the functional $F_{\lambda,\alpha}$ is not coercive on $H^1(\mathbb{B}^2)$ and the previous direct method does not hold any more. We shall recover the coercivity within the class of functions with vanishing barycenter (see \cref{lem. improved weak Moser on ball}).

    First, we show conformal invariance of  $F_{\lambda,\alpha}$ under the Mobius action $T_\phi$.
    \begin{lemma}\label{lem. conformal invariance of the functional}
        Let $u\in H^1(\mathbb{B}^2)$ and $\phi$ be a M\"obius transformation. Then
        \begin{align*}
            \int_{\mathbb{B}^2}e^{2T_\phi u}
        =\int_{\mathbb{B}^2} e^{2u},
        \quad
        \int_{\partial\mathbb{B}^2}e^{T_\phi u}
        =\int_{\partial\mathbb{B}^2} e^{u},\\
        \int_{\mathbb{B}^2}
        |\nabla T_\phi u|^2
        +2\int_{\partial\mathbb{B}^2}
        T_\phi u
        =\int_{\mathbb{B}^2} |\nabla u|^2
        +2\int_{\partial \mathbb{B}^2}u.
        \end{align*}
    \end{lemma}
    \begin{proof}
        Note that $\phi:\mathbb{B}^2\to \mathbb{B}^2$ is a conformal map, and
        \[
        e^{2T_\phi u}g_0
        =e^{2u\circ\phi}|\phi'|^2g_0
        =\phi^*\left(e^{2u}g_0\right).
        \]
        The first two identities then follows from change of variables.

        Denote $\log|\phi'|$ by $\varphi$, then we expand $T_\phi u$ and calculate
        \[
\begin{aligned}
    &\int_{\mathbb{B}^2}
        |\nabla T_\phi u|^2
        +2\int_{\partial\mathbb{B}^2}T_\phi u
        \\
    =&
    \int_{\mathbb {B}^2}
    |\nabla(u\circ\phi)|^2
    +2\int_{\mathbb {B}^2}
        \left\langle \nabla(u\circ\phi),\nabla\varphi\right\rangle        
    +\int_{\mathbb {B}^2}
    |\nabla\varphi|^2
    +2\int_{\partial\mathbb{B}^2}u\circ\phi
    +2\int_{\partial\mathbb {B}^2}\varphi.
\end{aligned}
\]
Since the Dirichlet energy is conformally invariant in dimension \(2\), the first term is
\[
    \int_{\mathbb{B}^2}
    |\nabla(u\circ\phi)|^2
    =
    \int_{\mathbb{B}^2}|\nabla u|^2.
\]
Next, notice that both \(g_0\) and \(\phi^*g_0=e^{2\varphi}g_0\) have Gauss curvature \(0\) and geodesic curvature $1$, so $-\Delta \varphi=0$ in $\mathbb{B}^2$ and $\frac{\partial\varphi}{\partial\nu}+1=e^\varphi$ on $\partial\mathbb{B}^2$.
Then by integration by parts,
\[
\begin{aligned}
    2\int_{\mathbb{B}^2}
        \left\langle \nabla(u\circ\phi),\nabla\varphi
        \right\rangle 
        +2\int_{\partial\mathbb{B}^2}u\circ\phi
    =&-2\int_{\mathbb{B}^2}(u\circ\phi)\Delta\varphi
    +2\int_{\partial\mathbb{B}^2}
    \left((u\circ\phi)\frac{\partial\varphi}{\partial\nu}
    +u\circ\phi\right)\\
    =&2\int_{\partial\mathbb{B}^2}
    (u\circ\phi)    e^\varphi 
    =2\int_{\partial\mathbb{B}^2} u.
\end{aligned}
\]
Substituting these two identities into the previous expansion,  we obtain
\[
\begin{aligned}
    \int_{\mathbb{B}^2}
        |\nabla T_\phi u|^2
        +2\int_{\partial\mathbb{B}^2}T_\phi u
        =\int_{\mathbb{B}^2}|\nabla u|^2
         +2\int_{\partial\mathbb{B}^2}u 
         +\int_{\mathbb{B}^2}|\nabla\varphi|^2
         +2\int_{\partial\mathbb{B}^2}\varphi .
\end{aligned}
\]
    Finally, a straightforward calculation yields
    \[
    \int_{\mathbb{B}^2}|\nabla\varphi|^2
    +2\int_{\partial\mathbb{B}^2}\varphi 
    =0,
    \]
and the proof finishes.
    \end{proof}
    Next, we fix the barycenter to be zero via a Hersch-type argument.
    \begin{lemma}
    \label{lem. Disk Hersch normalization}
        For any $u\in H^1(\mathbb{B}^2)$, there exists a M\"obius transformation $\phi$, such that
        \[
        \int_{\mathbb{B}^2}ze^{2T_\phi u(z)}=0,
        \]
        where we use the complex coordinate $z=x_1+ix_2$.
    \end{lemma}
    \begin{proof}
        For fixed $u\in H^1(\mathbb{B}^2)$, consider the continuous map
        \begin{align*}
            \Phi:\overline{\mathbb{B}^2}\to& \overline{\mathbb{B}^2}\\
            a\mapsto& \int_{\mathbb{B}^2}
            \phi_a(z)e^{2u(z)}.
        \end{align*}
        where  
        \begin{equation}
        \label{eq. Mobius transformation}
            \phi_a(z):=\frac{z+a}{1+\bar{a}z},\quad a\in\mathbb{B}^2,
        \end{equation}
        is a M\"obius transformation.
        Note that 
        \[
        \phi_a\equiv a, \quad \text{if }a\in\partial \mathbb{B}^2.
        \]
        It follows that
        \[
        \Phi|_{\partial\mathbb{B}^2}
        =\operatorname{Id}|
        _{\partial\mathbb{B}^2}.
        \]
        Therefore, a degree argument implies that $\Phi$ is a surjective map, and hence there exists $a\in\mathbb{B}^2$ such that $\Phi(a)=0$.
        Then, by change of variables,  we conclude that
        \[
        \int_{\mathbb{B}^2}ze^{2T_{\phi_{-a}}u}
        =\int_{\mathbb{B}^2}ze^{2u\circ \phi_{-a}}|\phi_{-a}'|^2
        =\int_{\mathbb{B}^2} \phi_a(z)e^{2u(z)}
        =0.
        \]
    \end{proof}

    Now we prove an improved version of \cref{eq. weak Moser on ball} under the addition vanishing barycenter constraint. Such improvement phenomenon could be traced back to the seminal work of Aubin \cite{Aub79}, and is further developed in many subsequent works, for example, \cite[Theorem 2.1]{CL91}, 
    \cite[Lemma 2.3]{Wan01},
    \cite[Proposition 2.2]{MR11}, 
    and \cite[Lemma 2.4]{L-SR16}.
    \begin{lemma}\label{lem. improved weak Moser on ball}
        For any $\epsilon>0$, there exists a constant $C_\epsilon>0$ such that, for any $u\in H^1(\mathbb{B}^2)$ with 
        \[
        \int_{\mathbb{B}^2}ze^{2u(z)}=0,
        \]
        there holds
        \[
        \log\left(
        \frac{1}{\pi}\int_{\mathbb{B}^2}e^{2u}
        \right)
        \leq \frac{1+\epsilon}{4\pi}\int_{\mathbb{B}^2}
        |\nabla u|^2
        +\frac{1}{\pi}
        \int_{\partial\mathbb{B}^2} u
        +C_\epsilon.
        \]
    \end{lemma}
    To pave the way for this improved inequality, we need a general concentration lemma for probability measures.
    \begin{lemma}\label{lem. concentration for measures}
        There exist two universal constants $\delta_0, d_0>0$ such that, for any probability measure $\mu$ with 
        \[
        \int_{\mathbb{B}^2}z\dd \mu=0,
        \]
        the following alternative holds
        \begin{enumerate}
            \item Either there holds
            \[
            \mu(B_{1-d_0})\geq \delta_0;
            \]
            \item or there exists $\Omega_1,\Omega_2\subset \overline{\mathbb{B}^2}$, such that $\operatorname{dist}(\Omega_1,\Omega_2)\geq d_0$ and
            \[
            \mu(\Omega_1)\geq \delta_0,
            \quad 
            \mu(\Omega_2)\geq \delta_0.
            \]
        \end{enumerate}
    \end{lemma}
    \begin{proof}
        Suppose for contradiction that the lemma is false. Then for $\delta_j:=\frac{1}{j}, d_j:=\frac{1}{j}$, there exists probability measures $\mu_j$ on $\mathbb{B}^2$ with $\int_{\mathbb{B}^2}z\dd \mu_j=0$,  satisfying
        \begin{equation}\label{eq. center mass is small}
            \mu_j(B_{1-d_j})
        <\delta_j;
        \end{equation}
        and, for any $\Omega_1,\Omega_2\subset\overline{\mathbb{B}^2}$ with $\operatorname{dist}(\Omega_1,\Omega_2)\geq d_j$, there holds
        \begin{equation}\label{eq. separated mass is small}
            \min\left\{
            \mu_j(\Omega_1),\mu_j(\Omega_2)
        \right\}<\delta_j.
        \end{equation}
        Since $\mu_j$ are probability measures, up to a subsequence, there exists a probability measure $\mu_\infty$ such that $\mu_j\overset{*}{\rightharpoonup}\mu_\infty$ in the sense that
        \[
        \int_{\mathbb{B}^2}f\dd \mu_j
        \to 
        \int_{\mathbb{B}^2}f\dd \mu_\infty,\quad 
        \forall f\in C(\mathbb{B}^2).
        \]
        For any $r\in(0,1)$, choose $j\gg 1$ such that $B_{1-d_j}\supset B_r$, then by \cref{eq. center mass is small} and the weak convergence of $\mu_j$, we have
        \[
        \mu_\infty(B_r)
        \leq\liminf_{j\to+\infty}
        \mu_j(B_r)=0.
        \]
        Therefore, $\mu_\infty$ is supported on $\partial\mathbb{B}^2$. 

        We claim that $\operatorname{supp}\mu_\infty$ is a singleton. Otherwise, say $p,q\in\operatorname{supp}\mu_\infty$ with $p\not=q$, then there exists two disjoint sets $\Omega_p,\Omega_q\subset\overline{\mathbb{B}^2}$ with $\mu_\infty(\Omega_p)>0,\mu_\infty(\Omega_p)>0$. Then we have
        \[
        \liminf_{j\to+\infty}
        \min\{\mu_j(\Omega_p),\mu_j(\Omega_q)\}
        \geq 
        \min\{\mu_\infty(\Omega_p),\mu_\infty(\Omega_q)\}
        >0.
        \]
        This contradicts with \cref{eq. separated mass is small}. Hence $\mu_\infty=\delta_a$ for some $a\in\partial\mathbb{B}^2$.
        
        It follows that
        \[
        0=\lim_{j\to+\infty}\int_{\mathbb{B}^2}
        z\dd \mu_j
        =\int_{\mathbb{B}^2}z\dd \mu_\infty
        =a\in\partial\mathbb{B}^2.
        \]
        This completes the proof.
    \end{proof}

    \begin{proof}[Proof of \cref{lem. improved weak Moser on ball}]
    {\bf Step 1:} We first assume that
    \[
    \int_{\mathbb{B}^2}u=0.
    \]
    For any $\epsilon>0$, let $V_\epsilon$ be the Neumann eigenspace corresponding to eigenvalues less than $\epsilon^{-1}$. Let $u_1\in L^{\infty}(\mathbb{B}^2)$ be the orthogonal projection of $u$ on $V_\epsilon$, and $u_2:=u-u_1\in H^1(\mathbb{B}^2)$. By definition and the Poincar\'e inequality, there holds
    \begin{align}
        ||u_1||_{L^{\infty}(\mathbb{B}^2)}
        \leq& C_\epsilon 
        ||u_1||_{L^2(\mathbb{B}^2)}
        \leq C_\epsilon ||\nabla u_1||_{L^2(\mathbb{B}^2)}
        \leq \epsilon\int_{\mathbb{B}^2}
        |\nabla u_1|^2
        +C_\epsilon,\label{eq. upper bound of u1}\\
        \int_{\mathbb{B}^2}|u_2|^2
        \leq& \epsilon 
        \int_{\mathbb{B}^2}|\nabla u_2|^2.
        \label{eq. L2 bound of u2}
    \end{align}

        By \cref{lem. concentration for measures}, the probability measure
        \[
        \dd\mu:=\frac{e^{2u(x)}}{\int_{\mathbb{B}^2}e^{2u}}\dd x
        \]
        satisfies the alternatives.

        {\bf Step 2:} For the first alternative, i.e.
        \[
        \int_{B_{1-d_0}}e^{2u}
        \geq \delta_0\int_{\mathbb{B}^2}e^{2u}.
        \]
        Let $\chi\in C_0^{\infty}(\mathbb{B}^2)$ be a cut-off function with $\chi\equiv 1$ in $B_{1-d_0}$. Then we have
        \begin{equation}\label{eq. area upper bound in first alternative}
            \delta_0 \int_{\mathbb{B}^2}e^{2u}
        \leq \int_{B_{1-d_0}}e^{2u}
        \leq e^{2||u_1||
        _{L^\infty(\mathbb{B}^2)}}
        \int_{\mathbb{B}^2}e^{2\chi u_2}.
        \end{equation}
        By Cauchy-Schwarz inequality and \cref{eq. L2 bound of u2}, there holds, for any $\eta>0$,
        \begin{equation}
        \label{eq. H01 norm upper bound}
            \int_{\mathbb{B}^2}|\nabla(\chi u_2)|^2
        \leq (1+\eta)
        \int_{\mathbb{B}^2}|\nabla u_2|^2
        +C_\eta\int_{\mathbb{B}^2}|u_2|^2
        \leq (1+\eta+C_\eta\epsilon)
        \int_{\mathbb{B}^2}|\nabla u_2|^2
        \end{equation}
        Applying \cref{eq. weak Moser for Dirichlet boundary value} to $\chi u_2$ and using \cref{eq. H01 norm upper bound}, we obtain
        \[
        \log\left(
        \frac{1}{\pi}
        \int_{\mathbb{B}^2}e^{2\chi u_2}
        \right)
        \leq \frac{1+\eta+C_\eta\epsilon}{4\pi}
        \int_{\mathbb{B}^2}|\nabla u_2|^2
        +C.
        \]
        Inserting this into \cref{eq. area upper bound in first alternative}, taking logarithm and using \cref{eq. upper bound of u1}, we get 
        \begin{align*}
            \log\left(
        \frac{1}{\pi}
        \int_{\mathbb{B}^2}e^{2u}
        \right)
        \leq& 2\epsilon\int_{\mathbb{B}^2}
        |\nabla u_1|^2
        +C_\epsilon
        +\frac{1+\eta+C_\eta\epsilon}{4\pi}
        \int_{\mathbb{B}^2}|\nabla u_2|^2
        +C\\
        \leq &\frac{1+\eta+C_\eta\epsilon}{4\pi}
        \int_{\mathbb{B}^2}|\nabla u|^2
        +C_\epsilon.
        \end{align*}
        First fixing $\eta\ll 1$ and then choosing $\epsilon\ll 1$, after renaming the parameters, we get
        \[
        \log\left(
        \frac{1}{\pi}
        \int_{\mathbb{B}^2}e^{2u}
        \right)
        \leq \frac{1+\epsilon}{4\pi}
        \int_{\mathbb{B}^2}|\nabla u|^2
        +C_\epsilon.
        \]    

        {\bf Step 3:} For the second alternative, i.e. there exists $\Omega_1,\Omega_2\subset \overline{\mathbb{B}^2}$, such that $\operatorname{dist}(\Omega_1,\Omega_2)\geq d_0$ and
        \[
        \int_{\Omega_1}e^{2u}
        \geq \delta_0
        \int_{\mathbb{B}^2}e^{2u},\quad
        \int_{\Omega_2}e^{2u}
        \geq \delta_0
        \int_{\mathbb{B}^2}e^{2u}.
        \]
        Let $\chi_1,\chi_2\in C^\infty(\mathbb{B}^2)$ be two functions with disjoint support and $\chi_i\equiv 1$ on $\Omega_i$. Then we have
        \begin{equation}\label{eq. area upper bound in second alternative}
            \delta_0 \int_{\mathbb{B}^2}e^{2u}
        \leq \int_{\Omega_i}e^{2u}
        \leq e^{2||u_1||
        _{L^\infty(\mathbb{B}^2)}}
        \int_{\mathbb{B}^2}
        e^{2\chi_i u_2}.
        \end{equation}
        Note that, by \cref{eq. L2 bound of u2},
        \begin{align}
            \int_{\operatorname{supp}\chi_i}
            |\nabla(\chi_i u_2)|^2
        \leq& (1+\eta)
        \int_{\operatorname{supp}\chi_i}
        |\nabla u_2|^2
        +C_\eta\int_{\mathbb{B}^2}|u_2|^2\notag
        \\
        \leq& (1+\eta)
        \int_{\operatorname{supp}\chi_i}|\nabla u_2|^2
        +C_\eta\epsilon
        \int_{\mathbb{B}^2}|\nabla u_2|^2.
        \label{eq. H01 norm upper bound in second}
        \end{align}
        By Sobolev-trace inequality, \cref{eq. L2 bound of u2}, \cref{eq. H01 norm upper bound in second} and Young's inequality, there holds
        \begin{align}
        \label{eq. upper bound of trace}
            \int_{\partial\mathbb{B}^2}\chi_i u_2
            \leq \sqrt{2\pi}
            ||\chi_2u_2||
            _{L^2(\partial \mathbb{B}^2)}
            \leq \sqrt{2\pi}
            ||\chi_2u_2||
            _{H^{1}(\mathbb{B}^2)}
            \leq \epsilon \int_{\mathbb{B}^2}
            |\nabla u_2|^2+C_\epsilon.
        \end{align}

        Applying \cref{eq. weak Moser on ball} to $\chi_i u_2$ and inserting \cref{eq. H01 norm upper bound in second} and \cref{eq. upper bound of trace} into it, we obtain
        \[
        \log\left(
        \frac{1}{\pi}
        \int_{\mathbb{B}^2}e^{2\chi_i u_2}
        \right)
        \leq 
        \frac{1+\eta}{2\pi}
        \int_{\operatorname{supp}\chi_i}|\nabla u_2|^2
        +\left(\frac{C_\eta\epsilon}{2\pi}
        +\frac{\epsilon}{\pi}
        \right)
        \int_{\mathbb{B}^2}|\nabla u_2|^2
            +C_\eta+C_\epsilon.
        \]
        Inserting this into \cref{eq. area upper bound in second alternative}, taking logarithm and using \cref{eq. upper bound of u1}, we get 
        \[
        \log\left(
        \frac{1}{\pi}
        \int_{\mathbb{B}^2}e^{2u}
        \right)
        \leq
        2\epsilon\int_{\mathbb{B}^2}
        |\nabla u_1|^2
        +\frac{1+\eta}{2\pi}
        \int_{\operatorname{supp}\chi_i}|\nabla u_2|^2
        +\left(\frac{C_\eta\epsilon}{2\pi}
        +\frac{\epsilon}{\pi}
        \right)
        \int_{\mathbb{B}^2}|\nabla u_2|^2
        +C_\eta+C_\epsilon.
        \]
        By adding the inequalities corresponding to $i=1$ and $i=2$, we derive
        \begin{align*}
            2\log\left(
        \frac{1}{\pi}
        \int_{\mathbb{B}^2}e^{2u}
        \right)
        &\leq
        4\epsilon\int_{\mathbb{B}^2}
        |\nabla u_1|^2
        +\frac{1+\eta}{2\pi}
        \int_{\mathbb{B}^2}|\nabla u_2|^2
        +\left(\frac{C_\eta\epsilon}{\pi}
        +\frac{2\epsilon}{\pi}
        \right)
        \int_{\mathbb{B}^2}|\nabla u_2|^2
        +C_\eta+C_\epsilon\\[3mm]
        &\leq 
        \left(
        \frac{1+\eta}{2\pi}
        +\frac{(C_\eta+2)\epsilon}{\pi}
        \right)
        \int_{\mathbb{B}^2}|\nabla u|^2
        +C_\eta+C_\epsilon.
        \end{align*}
        First fixing $\eta\ll 1$ and then choosing $\epsilon\ll 1$, after renaming the parameters, we get
        \[
        \log\left(
        \frac{1}{\pi}
        \int_{\mathbb{B}^2}e^{2u}
        \right)
        \leq \frac{1+\epsilon}{4\pi}
        \int_{\mathbb{B}^2}|\nabla u|^2
        +C_\epsilon.
        \] 

        {\bf Step 4:} We have proved that, for any $\epsilon>0$, there exists $C_\epsilon>0$, such that if $\int_{\mathbb{B}^2}ze^{2u(z)}=0$ and
        $\int_{\mathbb{B}^2}u=0$, then  
        \[
        \log\left(
        \frac{1}{\pi}
        \int_{\mathbb{B}^2}e^{2u}
        \right)
        \leq \frac{1+\epsilon}{4\pi}
        \int_{\mathbb{B}^2}|\nabla u|^2
        +C_\epsilon.
        \]
        By translation, this is equivalent to
        \begin{equation}
        \label{eq. improved weak Moser on ball with interior mean}
            \log\left(
        \frac{1}{\pi}
        \int_{\mathbb{B}^2}e^{2u}
        \right)
        \leq \frac{1+\epsilon}{4\pi}
        \int_{\mathbb{B}^2}|\nabla u|^2
        +\frac{2}{\pi}\int_{\mathbb{B}^2}u
        +C_\epsilon,\quad 
        \forall u\in H^1(\mathbb{B}^2)
        \text{ with } 
        \int_{\mathbb{B}^2}ze^{2u(z)}=0.
        \end{equation}
        
        Now we claim that
        \begin{equation}
        \label{eq. difference between interior mean and boundary mean}
            \left|
        \frac{2}{\pi}\int_{\mathbb{B}^2}u
        -\frac{1}{\pi}
        \int_{\partial\mathbb{B}^2}u
        \right|
        \leq C
        \left(\int_{\mathbb{B}^2}|\nabla u|^2
        \right)^{\frac{1}{2}},\quad 
        \forall u\in H^1(\mathbb{B}^2).
        \end{equation}
        Otherwise, there exists a sequence $u_j\in H^1(\mathbb{B}^2)$ such that
        \begin{equation}
        \label{eq. blow up condition}
            \int_{\mathbb{B}^2}|\nabla u_j|^2=1,\quad 
        \left|
        \frac{2}{\pi}\int_{\mathbb{B}^2}u_j
        -\frac{1}{\pi}
        \int_{\partial\mathbb{B}^2}u_j
        \right|
        \geq j.
        \end{equation}
        After a translation, we assume that $\int_{\mathbb{B}^2}u_j=0$. Then by \cref{eq. blow up condition}, Cauchy-Schwarz inequality, Sobolev-trace inequality and Poincar\'e inequality, we have
        \[
        j
        \leq \frac{1}{\pi}\int_{\partial \mathbb{B}^2}
        |u_j|
        \leq 2\left(
        \frac{1}{2\pi}
        \int_{\partial\mathbb{B}^2}|u_j|^2
        \right)^{\frac{1}{2}}
        \leq C||u_j||_{H^1(\mathbb{B}^2)}
        \leq C||\nabla u_j||_{L^2(\mathbb{B}^2)}
        =C.
        \]
        This is a contradiction. Combining \cref{eq. improved weak Moser on ball with interior mean} with \cref{eq. difference between interior mean and boundary mean} completes the proof.
    \end{proof}

    \begin{proof}[Proof of \cref{thm. sharp Sobolev-trace inequality} when $1\leq \alpha<2$]
        By \cref{lem. conformal invariance of the functional}, \cref{lem. Disk Hersch normalization} and the translation-invariance of $F_{1,\alpha}$, it suffices to minimize the functional over
        \[
        \mathcal{S}:=
        \left\{u\in H^1(\mathbb{B}^2):
        \int_{\partial\mathbb{B}^2}u=0,\ 
        \int_{\mathbb{B}^2}ze^{2u(z)}=0
        \right\}.
        \]
        By Jensen's inequality, there holds
        \[
        \log\left(\frac{1}{2\pi}
        \int_{\partial\mathbb{B}^2}e^u
        \right)
        \geq \int_{\partial\mathbb{B}^2}
        \frac{u}{2\pi}
        =0.
        \]
        Combining this with \cref{lem. improved weak Moser on ball}, we derive
        \begin{align*}
            F_{1,\alpha}(u)
        &= \frac{1}{4\pi}
        \int_{\mathbb{B}^2}|\nabla u|^2
        +\frac{1}{2\pi}
        \int_{\partial \mathbb{B}^2}u
        -\frac{\alpha}{2}
        \log\left(
        \frac{1}{\pi}\int_{\mathbb{B}^2}e^{2u}
        \right)
        +(\alpha-1)
        \log\left(
        \frac{1}{2\pi}
        \int_{\partial\mathbb{B}^2}e^{u}
        \right)\\[3mm]
        & \geq \frac{2-(1+\epsilon)\alpha}{8\pi}
        \int_{\mathbb{B}^2}|\nabla u|^2
        -C_\epsilon.
        \end{align*}
        Therefore, by choosing $\epsilon$ satisfying $2-(1+\epsilon)\alpha>0$, we know that $F_{1,\alpha}$ is coercive on $\mathcal{S}$. By \cref{lem. compactness in H1}, the same arguments as that in \cref{App results} shows that there exists a minimizer $u_\infty$ in $\mathcal{S}$. 
        
        We claim that this constrained minimizer is  an unconstrained critical point of $F_{1,\alpha}$. Indeed, for any variation $u_\infty+tv$, by \cref{lem. conformal invariance of the functional}, \cref{lem. Disk Hersch normalization} and the translation-invariance of $F_{1,\alpha}$,  there exists exists a smooth family of M\"obius transformations $\phi_t$ such that 
        \[
        T_{\phi_t}(u_\infty+tv)\in\mathcal{S},
        \]
        and
        \[
        F_{1,\alpha}(T_{\phi_t}(u_\infty+tv))
        =F_{1,\alpha}(u_\infty+tv).
        \]
        Differentiating it at $t=0$ yields that $u_\infty$ satisfies the unconstrained Euler-Lagrange equation \cref{eq. equation of minimizer} with $\lambda=1$.
        The rest of proof is the same as that in \cref{App results}.
    \end{proof}
    Finally, we show the sharpness for the range of $\alpha$.
\begin{proposition}
\label{prop. sharpness of alpha}
\
    \begin{enumerate}
        \item For $\alpha=2$, there holds
        \[
        \inf_{u\in H^1(\mathbb{B}^2)}F_{1,2}(u)
        = -1.
        \]
        The infimum is not achieved.
        \item Extend \cref{eq. def of functional F} formally to all
        $\lambda,\alpha\in\mathbb{R}$, then
        \[
        \inf_{u\in H^1(\mathbb{B}^2)}
        F_{\lambda,\alpha}(u)>-\infty\quad
        \Longleftrightarrow\quad 
        \lambda\leq 1 \text{ and } 
        \alpha\leq 2.
        \]
        In these cases, the infimum equals
        \[
        -\frac{\alpha}{2}
        -\left(1-\frac{\alpha}{2}\right)
        \log\left(1-\frac{\alpha}{2}\right).
        \]
    \end{enumerate}
    \end{proposition}
    \begin{proof}
        By \cref{thm. sharp Sobolev-trace inequality}, for $\alpha<2$, we have
        \[
        F_{1,\alpha}(u)
        \geq -\frac{\alpha}{2}
        -\left(1-\frac{\alpha}{2}\right)
        \log\left(1-\frac{\alpha}{2}\right).
        \]
        Letting $\alpha$ go to $2$ shows that $F_{1,2}(u)\geq -1$. The sharpness follows from \cref{lem. F of wq}:
        \[
        F_{1,2}(w_q)=-1+\frac{1}{1+q}\to -1,\quad \text{as }q\to+\infty.
        \]
        On the other hand, by \cref{cor. stability inequality}, there holds
        \[
        F_{1,2}(u)\geq -1+
        \frac{\left(
       \int_{\partial\mathbb{B}^2}e^{u}
       \right)^2}
       {4\pi\int_{\mathbb{B}^2}e^{2u}}.
        \]
        Hence the infimum is not achieved.

        \vspace{1em}

        For $\lambda\leq 1$ and $\alpha\leq 2$, note that
        \begin{equation}
        \label{eq. representation of general functional}
            F_{\lambda,\alpha}(u)
        =F_{1,\alpha}(u)
        +(1-\lambda)
        \left(
        \log\left(\frac{1}{2\pi}\int_{\partial\mathbb{B}^2}e^u
        \right)
        -\frac{1}{2\pi}
        \int_{\partial\mathbb{B}^2}u
        \right).
        \end{equation}
        Therefore, by \cref{thm. sharp Sobolev-trace inequality} and Jensen's inequality, we get
        \[
        F_{\lambda,\alpha}
        \geq
        -\frac{\alpha}{2}
        -\left(1-\frac{\alpha}{2}\right)
        \log\left(1-\frac{\alpha}{2}\right).
        \]
        This lower bound is achieved by $w_{\frac{\alpha}{2-\alpha}}$ if $\alpha<2$.

        For $\alpha>2$, \cref{lem. F of wq} shows that
        \[
        F_{\lambda,\alpha}(w_q)
        =-\frac{\alpha-2}{2}\log (1+q)
        -\frac{q}{1+q}\to -\infty,
        \quad \text{as }q\to+\infty.
        \]

        For $\lambda>1$, consider a fixed $q>-1$ and
        \[
        T_{\phi_a}w_q
        :=w_q\circ\phi_a +\log |\phi_a'|,
        \]
        where $\phi_a$ is the Mobius transformation defined in \cref{eq. Mobius transformation}. Note that
        \[
        T_{\phi_a}w_q(x)=\log\frac{1-|a|^2}{1+|a|^2+2\langle x,a\rangle} \quad 
        \text{on }\partial\mathbb{B}^2,
        \]
        A residue calculation yields
        \[
        \log\left(\frac{1}{2\pi}\int_{\partial\mathbb{B}^2}e^u
        \right)
        =0,\quad 
        \frac{1}{2\pi}
        \int_{\partial\mathbb{B}^2}u
        =\log (1-|a|^2).
        \]
        Therefore, by \cref{eq. representation of general functional}, we have
        \[
        F_{\lambda,\alpha}(T_{\phi_a}w_q)
        =F_{1,\alpha}(T_{\phi_a}w_q)
        +(\lambda-1)\log(1-|a|^2)
        \to -\infty,\quad
        \text{as }a\to \partial\mathbb{B}^2.
        \]
    \end{proof}
    \begin{remark}
        The two thresholds $\lambda\leq 1$ and $\alpha\leq 2$ correspond to different loss of compactness phenomena.
        For $\alpha>2$, the measure
        \[
        \frac{e^{2w_q(x)}}{
        \int_{\mathbb{B}^2}e^{2w_q}
        }
        \dd x
        \]
        converges weakly to the Dirac measure $\delta_0$ as $q\to+\infty$, while the boundary measure remains the uniform measure
        \[
        \frac{e^{w_q(x)}}{
        \int_{\partial \mathbb{B}^2}e^{w_q}
        }
        \dd \theta
        =\frac{1}{2\pi}\dd \theta.
        \]
        For $\lambda>1$ and a fixed $q>-1$, the measure
        \[
        \frac{e^{2T_{\phi_a}w_q(x)}}{
        \int_{\mathbb{B}^2}e^{2T_{\phi_a}w_q}
        }
        \dd x
        \]
        converges weakly to $\delta_{-\xi}$ as 
        $a\to\xi\in\partial\mathbb{B}^2$, and the boundary measure
        \[
        \frac{e^{T_{\phi_a}w_q(x)}}{
        \int_{\partial\mathbb{B}^2}
        e^{T_{\phi_a}w_q}
        }
        \dd \theta
        =\frac{|\phi_a'|}{2\pi}
        \]
        also converges weakly to $\delta_{-\xi}$ as 
        $a\to\xi\in\partial\mathbb{B}^2$.
    \end{remark}

    \section{Deficit estimate for Lebedev-Milin inequality}
    \label{sec. stability ineq for LM ineq}
    As a corollary of \cref{thm. sharp Sobolev-trace inequality}, we obtain a deficit estimate for the classical Lebedev-Milin inequality.
    Here we present a more precise form of \cref{cor. stability for LM inequality}.
    \begin{corollary}
    \label{cor. stability inequality}
       For $u\in H^1(\mathbb{B}^2)$, define the isoperimetric ratio
       \[
       I(u):=
       \frac{\left(
       \int_{\partial\mathbb{B}^2}e^{u}
       \right)^2}
       {4\pi\int_{\mathbb{B}^2}e^{2u}}
       =\frac{\left(\frac{1}{2\pi}
        \int_{\partial\mathbb{B}^2}e^{u}
        \right)^2}
        {\frac{1}{\pi}
        \int_{\mathbb{B}^2}e^{2u}}.
       \]
       Then there holds
       \begin{equation}\label{eq. stability of Lebedev-Milin inequality}
           \frac{1}{4\pi}
        \int_{\mathbb{B}^2}|\nabla u|^2
        +\frac{1}{2\pi}
        \int_{\partial \mathbb{B}^2}u
        -\log\left(
        \frac{1}{2\pi}
        \int_{\partial\mathbb{B}^2}e^{u}
        \right)
        \geq \varphi(I(u)), \qquad \forall u\in H^1(\mathbb{B}^2),
       \end{equation}
       where $\varphi(t):=t-1-\log t\geq0$ for all $t>0$.
       Moreover, equality holds if and only if 
       \begin{equation}
       \label{eq. minimizer in sharp LM inequality}
           u=T_\phi w_q+c,
       \end{equation}
       where $\phi:\mathbb{B}^2\to \mathbb{B}^2$ is a M\"obius transformation and $c\in\mathbb{R}$ is a constant.

       More generally, for $\alpha\in(-\infty,2]$, we have a deficit estimate for the sharp inequality of $F_{1,\alpha}$:
       \[
       F_{1,\alpha}(u) 
       \geq 
       \begin{cases}
           -\frac{\alpha}{2}
        -\left(1-\frac{\alpha}{2}\right)
        \log\left(1-\frac{\alpha}{2}\right)
        +\left(1-\frac{\alpha}{2}\right)
        \varphi\left(\frac{I(u)}{1-\frac{\alpha}{2}}\right), &\text{if }\alpha<2,\\
        -1+I(u), &\text{if }\alpha=2.
       \end{cases}
       \]
    \end{corollary}
    \begin{proof}
        Note that, by \cref{thm. sharp Sobolev-trace inequality},
        \[
        F_{1,0}(u)
        +\frac{\alpha}{2}\log I(u)
        =F_{1,\alpha}(u)
        \geq -\frac{\alpha}{2}
        -\left(1-\frac{\alpha}{2}\right)
        \log\left(1-\frac{\alpha}{2}\right).
        \]
        Then  we have
        \begin{align*}
            \frac{1}{4\pi}
        \int_{\mathbb{B}^2}|\nabla u|^2
        +\frac{1}{2\pi}
        \int_{\partial \mathbb{B}^2}u
        -\log\left(
        \frac{1}{2\pi}
        \int_{\partial\mathbb{B}^2}e^{u}
        \right)
        \geq -\frac{\alpha}{2}
        -\left(1-\frac{\alpha}{2}\right)
        \log\left(1-\frac{\alpha}{2}\right)
        -\frac{\alpha}{2}\log I(u).
        \end{align*}
        Regard the right hand side of the inequality as a function of $\alpha$. Then a direct calculation shows that its unique maximum occurs at $\alpha=2-2I(u)$, and the maximum value is exactly
        \[
        I(u)-1-\log I(u)=\varphi(I(u)).
        \]
        This proves \cref{eq. stability of Lebedev-Milin inequality}. If the equality holds for some $u$, tracing back to the proof, we get 
        \[
        \alpha=2-2I(u),
        \]
        and the infimum of $F_{1,\alpha}$ is also achieved by $u$. Then by \cref{thm. sharp Sobolev-trace inequality}, we know $u$ must be of the form \cref{eq. minimizer in sharp LM inequality}. By \cref{lem. F of wq}, it is straightforward to check that they indeed achieve the equality in \cref{eq. stability of Lebedev-Milin inequality}.

        Finally, for $\alpha\in(-\infty,2)$, by \cref{eq. stability of Lebedev-Milin inequality} we have
        \begin{align*}
            F_{1,\alpha}(u) 
       +\frac{\alpha}{2}
        & +\left(1-\frac{\alpha}{2}\right)
        \log\left(1-\frac{\alpha}{2}\right)\\
        &= F_{1,0}(u)+\frac{\alpha}{2}\log I(u)
        +\frac{\alpha}{2}
        +\left(1-\frac{\alpha}{2}\right)
        \log\left(1-\frac{\alpha}{2}\right)\\
        & \geq I(u)
        -\left(1-\frac{\alpha}{2}\right)
        -\left(1-\frac{\alpha}{2}\right)\log I(u)
        +\left(1-\frac{\alpha}{2}\right)
        \log \left(1-\frac{\alpha}{2}\right) \\
        & =\left(1-\frac{\alpha}{2}\right)
        \varphi\left(
        \frac{I(u)}{1-\frac{\alpha}{2}}\right)
        \end{align*}
        Letting $\alpha$ go to $2$ yields
        \[
        F_{1,2}(u)
        \geq -1+I(u).
        \]
    \end{proof}
    Now we prove \cref{cor. sharp Carleson-Chang ineq}, which is also a sharp improvement of \cref{eq. weak Moser for Dirichlet boundary value}.
    
    \begin{proof}[Proof of \cref{cor. sharp Carleson-Chang ineq}]
        For $u\in H_0^1(\mathbb{B}^2)$, we have
        \[
        I(u)=\frac{\pi}{\int_{\mathbb{B}^2}e^{2u}}.
        \]
        Then the inequality follows directly from \cref{eq. stability of Lebedev-Milin inequality}.
        The rigidity part follows from \cref{cor. stability inequality}.
    \end{proof}

\addcontentsline{toc}{chapter}{\numberline{}Bibliography}
\bibliographystyle{amsalpha}
\bibliography{refOnofri}
\end{document}